\documentclass[11pt]{amsart}

\usepackage{amscd}
\usepackage{amsmath, amssymb}
\usepackage{amsfonts}
\newcommand{\de}{\partial}
\newcommand{\db}{\overline{\partial}}

\newcommand{\ddbar}{\sqrt{-1} \partial \overline{\partial}}
\newcommand{\Ric}{\mathrm{Ric}}
\newcommand{\ov}[1]{\overline{#1}}
\newcommand{\mn}{\sqrt{-1}}

\newcommand{\tr}[2]{\textrm{tr}_{#1}{#2}}
\newcommand{\ti}[1]{\tilde{#1}}
\newcommand{\vp}{\varphi}

\newcommand{\ve}{\varepsilon}

\renewcommand{\leq}{\leqslant}
\renewcommand{\geq}{\geqslant}
\renewcommand{\le}{\leqslant}
\renewcommand{\ge}{\geqslant}

\newcommand{\be}{\begin{equation}}
\newcommand{\ee}{\end{equation}}

\begin{document}
\newcounter{remark}
\newcounter{theor}
\setcounter{remark}{0}
\setcounter{theor}{1}
\newtheorem{claim}{Claim}
\newtheorem{theorem}{Theorem}[section]
\newtheorem{lemma}[theorem]{Lemma}
\newtheorem{corollary}[theorem]{Corollary}
\newtheorem{proposition}[theorem]{Proposition}
\newtheorem{question}{question}[section]
\newtheorem{defn}{Definition}[theor]
\numberwithin{equation}{section}

\title[The Monge-Amp\`ere equation]{The Monge-Amp\`ere equation for non-integrable almost complex structures}
\author[J. Chu]{Jianchun Chu}
\address{School of Mathematical Sciences, Peking University, Yiheyuan Road 5, Beijing, P.R.China, 100871}
\author[V. Tosatti]{Valentino Tosatti}
\address{Department of Mathematics, Northwestern University, 2033 Sheridan Road, Evanston, IL 60208}
\author[B. Weinkove]{Ben Weinkove}
\address{Department of Mathematics, Northwestern University, 2033 Sheridan Road, Evanston, IL 60208}

\begin{abstract}
We show existence and uniqueness of solutions to the Monge-Amp\`ere equation on compact almost complex manifolds with non-integrable almost complex structure.
\end{abstract}
\maketitle

\section{Introduction}

Yau's Theorem for compact K\"ahler manifolds  \cite{Y1} states that one can prescribe the volume form of a K\"ahler metric within a given K\"ahler class.  This result, proved forty years ago, occupies a central place in the theory of K\"ahler manifolds, with wide-ranging applications in geometry and mathematical physics.

More precisely, Yau's Theorem is as follows.  Let $(M^{2n}, \omega, J)$ be a compact K\"ahler manifold of complex dimension $n$, where $\omega$ denotes the K\"ahler form and $J$ the complex structure.  Then given a smooth volume form $e^F \omega^n$ on $M$ there exists a unique smooth function $\varphi$ satisfying
\begin{equation} \label{mayau}
\begin{split}
(\omega + \sqrt{-1} \partial \bar \partial \varphi )^n = {} & e^{F} \omega^n, \\ \quad \omega + \sqrt{-1} \partial \bar \partial \varphi > 0, \quad {} &\sup_M \varphi=0.
\end{split}
\end{equation}
as long as $F$ satisfies the necessary normalization condition that $\int_M e^F \omega^n = \int_M \omega^n$.  The equation (\ref{mayau}) is known as the complex Monge-Amp\`ere equation for K\"ahler manifolds.

There has been great interest in extending Yau's Theorem to non-K\"ahler settings.  Recall that  a K\"ahler metric is a positive definite real $(1,1)$ form $\omega$, namely a \emph{Hermitian metric}, which satisfies
$$d \omega=0.$$
One extension of Yau's Theorem, initiated by Cherrier \cite{Ch} in the 1980s, is to remove this closedness condition.  This was carried out in full generality in \cite{TW1} where it was shown that (\ref{mayau}) has a unique solution for $\omega$ Hermitian, up to adding a (unique) constant to $F$ (see also \cite{TW0, GL}, as well as \cite{Bl2, Chu1, DK, KN,KN2,N,PSS,S, TW2,Zh,ZZ} for later developments).  A different recent extension on complex manifolds \cite{STW}, confirming a conjecture of Gauduchon \cite{Ga}, is that one can prescribe the volume form of a
 \emph{Gauduchon metric} (satisfying $\partial \bar \partial (\omega^{n-1})=0$).  In this case, (\ref{mayau}) is replaced by a Monge-Amp\`ere type equation for $(n-1)$-plurisubharmonic functions \cite{HL0, FWW1, FWW2, Po, TW4, TW5}.  Within this circle of ideas remain some open conjectures, such as a version of Yau's Theorem for balanced metrics (satisfying $d (\omega^{n-1})=0$) \cite{FWW1, TW4, TW5}.  This is related to the some questions about the Strominger system of mathematical physics, see e.g. \cite{FY, LYa, PPZ} and the references therein.

A different type of extension of Yau's Theorem is to the case when $J$ is a \emph{non-integrable} almost complex structure, and this is subject of the current paper.    Around twenty years ago, Gromov posed the following problem to P. Delan\"oe  \cite{De}: let $(M^{2n},\omega)$ be a compact symplectic manifold, $J$ an almost complex structure compatible with $\omega$ and $F$ a smooth function on $M$ with $\int_M e^F\omega^n=\int_M\omega^n$. Can one find a smooth function $\vp$ on $M$ such that $\omega+d(Jd\vp)$ is a symplectic form taming $J$ and satisfying
\begin{equation}\label{grom}
(\omega+d(Jd\vp))^n=e^F\omega^n?
\end{equation}
However, Delan\"oe \cite{De} showed that when $n=2$ the answer to this question is negative, and this was later extended to all dimensions by Wang-Zhu \cite{WZ}. The key ingredient of their results is the construction of a smooth function $\vp_0$ such that $\omega+d(Jd\vp_0)$ is on the boundary of the set of taming symplectic forms (so its $(1,1)$ part is semipositive definite but not strictly positive), and yet $(\omega+d(Jd\vp_0))^n>0$. This is possible because in this case the $(2,0)+(0,2)$ part of $d(Jd\vp_0)$ contributes a strictly positive amount.

This indicates that the problem with Gromov's suggestion is that the $2$-form $d(Jd\vp)$ is in general not of
type $(1,1)$ with respect to $J$, due to the fact that $J$ may not be integrable (in fact, $d(Jd f)$ is of type $(1,1)$ for all functions $f$ if and only if $J$ is integrable). Its $(1,1)$ part (up to an unimportant factor of $2$) will be denoted by
$$\mn\de\db\vp=\frac{1}{2}(d(Jd\vp))^{(1,1)},$$
which agrees with the standard notation when $J$ is integrable (see also section \ref{sectbas} for more explanations). This quantity was apparently first explicitly considered in \cite{Ha}.

We show that an analogue of Gromov's problem {\em does} hold after replacing $d(Jd\vp)$ with $\ddbar\vp$.   In fact we do not even require the manifold to be symplectic.  We obtain the following result for almost complex manifolds equipped with an almost Hermitian metric $\omega$.

Our main theorem is the following:

\begin{theorem}\label{main}
Let $(M, \omega, J)$ be a compact almost Hermitian manifold of real dimension $2n$.  Given a smooth function $F$ there exists a unique pair $(\varphi, b)$ where $\varphi \in C^{\infty}(M)$ and $b \in \mathbb{R}$, solving
\begin{equation} \label{ma}
\begin{split}
(\omega + \sqrt{-1} \partial \bar \partial \varphi )^n = {} & e^{F+b} \omega^n, \\ \quad \omega + \sqrt{-1} \partial \bar \partial \varphi > 0, \quad {} &\sup_M \varphi=0.
\end{split}
\end{equation}
\end{theorem}

Namely, this says that the main result of \cite{TW1} holds even when $J$ is a non-integrable almost complex structure.  The equation (\ref{ma}) was first considered by Harvey-Lawson \cite{HL} and Pli\'s \cite{P} in the setting of the Dirichlet problem for $J$-pseudoconvex domains (see Remark 5 below). We note here that there has been a renewed interest recently in the theory of $J$-plurisubharmonic functions on almost complex manifolds, i.e. functions $\vp$ for which $\ddbar\vp>0$ (see e.g. \cite{CGS, Ha, HL, HLP, Pa,P,P2}), and our main theorem fits well into this picture.

The key ingredients for proving Theorem \ref{main} are the following new {\em a priori} estimates:

\pagebreak[3]
\begin{theorem}\label{main2}
Let $(M, \omega, J)$ be a compact almost Hermitian manifold of real dimension $2n$.  Given smooth functions $F$ and $\varphi$ satisfying
\begin{equation} \label{ma2}
\begin{split}
(\omega + \sqrt{-1} \partial \bar \partial \varphi )^n = {} & e^{F} \omega^n, \\ \quad \omega + \sqrt{-1} \partial \bar \partial \varphi > 0, \quad {} &\sup_M \varphi=0,
\end{split}
\end{equation}
then there are uniform $C^\infty$ {\em a priori} estimates on $\vp$ depending only on $(M,\omega,J)$ and bounds for $F$.
\end{theorem}

We make now a few remarks about our results:

\bigskip
\noindent
1.
  We can also interpret the Monge-Amp\`ere equation \eqref{ma} in terms of Ricci curvature forms. More precisely,
we can associate to an almost Hermitian metric $\omega$  a {\em canonical connection} (see e.g. \cite{TWY}), whose curvature form can be expressed as a skew-Hermitian matrix of $2$-forms $\{\Omega^i_j\}$ ($1\leq i,j\leq n$) in any local unitary frame, and defining
$$\Ric(\omega):=\mn \sum_{i=1}^n\Omega^i_i,$$
we obtain a globally defined closed real $2$-form, which is cohomologous to the first Chern class $2\pi c_1(M,J)$ in $H^2(M,\mathbb{R})$. If $\vp$ is a smooth function with $\ti{\omega}:=\omega+\ddbar\vp>0$  then \eqref{ma} holds for some constant $b$ if and only if
\begin{equation} \label{ricci}
\Ric(\ti{\omega})=\Ric(\omega)- \frac{1}{2}d(Jd F),
\end{equation}
as follows easily from \cite[(3.16)]{TWY}.
Note however that in general there are representatives of the first Chern class $2\pi c_1(M,J)$ in $H^2(M,\mathbb{R})$ which cannot be written in
the form $\Ric(\omega)- \frac{1}{2}d(Jd F)$ for any $\omega$ and $F$, even when $J$ is integrable (cf. \cite[Corollary 2]{TW0}). Note also that in general $\Ric(\omega)$ is different from the Riemannian Ricci curvature of the Levi-Civita connection of the Riemannian metric induced by $\omega$. For example, when $(M,J)$ is the Kodaira-Thurston $4$-manifold with a certain explicit almost complex structure, then one can construct an explicit almost Hermitian metric $\omega$ with $\Ric(\omega)=0$ (see e.g. \cite[Theorem 4.1]{TW35}), and yet $M$ does not admit any Ricci-flat Riemannian metrics. Indeed, four-dimensional Einstein manifolds satisfy $\chi(M)\geq 0$ with equality only when the metric is flat \cite[6.32]{Bes}. The Kodaira-Thurston manifold is a $T^2$-bundle over $T^2$ and therefore has $\chi(M)=0$, so if $M$ admitted an Einstein metric, $M$ would be finitely covered by a torus, which is not the case.

\bigskip
\noindent
2. Suppose now that $(M^{2n},\omega, J)$ is a compact \emph{symplectic} manifold with $J$  an almost complex structure tamed by $\omega$.  In this case  $\omega^{(1,1)}$ is an almost Hermitian metric.
Given a smooth function $F$ we obtain a pair $(\vp,b)$ solving
\begin{equation} \label{mas}
\begin{split}
(\omega^{(1,1)} + \sqrt{-1} \partial \bar \partial \varphi )^n = {} & e^{F+b} \omega^n, \\ \quad \omega^{(1,1)} + \sqrt{-1} \partial \bar \partial \varphi > 0, \quad {} &\sup_M \varphi=0.
\end{split}
\end{equation}
If we define $\ti{\omega}=\omega+\frac{1}{2}d(Jd\vp),$
then a simple calculation (cf. \cite[Proposition 5]{De}) shows that $\ti{\omega}$ is a symplectic form taming $J$ and
\begin{equation}\label{subs}
\ti{\omega}^n\geq (\omega^{(1,1)}+\ddbar\vp)^n=e^{F+b}\omega^n.
\end{equation}

\bigskip
\noindent
3.  A different extension of Yau's Theorem to the symplectic setting was proposed by Donaldson \cite{Do}.  He considered a symplectic $4$-manifold $(M, \omega, J)$ with $J$ a tamed almost complex structure, and conjectured that, given a smooth function $F$, and given a symplectic form $\tilde{\omega}$ cohomologous to $\omega$, compatible with $J$ and solving
$$\tilde{\omega}^2 = e^F \omega^2,$$
then $\ti{\omega}$ should have uniform $C^\infty$ {\em a priori} estimates.
Donaldson showed that this conjecture and various extensions of it would have consequences for symplectic topology.   This problem is fundamentally different from our setup, since here the difference $\tilde{\omega}-\omega$ is not given by a simple operator applied to a function  $\varphi$.  Donaldson's conjecture remains open in general, but it is known to hold in some special cases \cite{W, TWY, TW3, TW35, FLSV, BFV}.

\bigskip
\noindent
4. One can also consider the parabolic Monge-Amp\`ere equation
$$\frac{\de \vp_t}{\de t}=\log\frac{(\omega + \ddbar\vp_t)^n}{\omega^n}-F,\quad \vp_0=0,$$
where we require that $\omega+\ddbar\vp_t>0$. Analogously to the result proved by Cao \cite{Ca} in the K\"ahler case and Gill \cite{G} in the Hermitian case, we expect that the techniques developed in this paper can be used to show that a smooth solution $\vp_t$ exists for all $t\geq 0$, and after suitable normalization converges smoothly to the solution of \eqref{ma}.

Similarly, we can consider the equation
$$\frac{\de \vp_t}{\de t}=\log\frac{(\omega-t\Ric(\omega)^{(1,1)} + \ddbar\vp_t)^n}{\omega^n},\quad \vp_0=0,$$
where now we require that $\omega-t\Ric(\omega)^{(1,1)} + \ddbar\vp_t>0$. This is in fact equivalent to the evolution equation for almost Hermitian forms
$$\frac{\de\omega_t}{\de t}=-\Ric(\omega_t)^{(1,1)},\quad \omega_0=\omega,$$
which is the K\"ahler-Ricci flow if $J$ is integrable and $\omega_0$ is K\"ahler, and the Chern-Ricci flow \cite{G, TW2} if $J$ is integrable. Again, using the techniques developed in this paper one should be able to characterize the maximal existence time of this flow, exactly as in \cite{TZ, TW2}.

\bigskip
\noindent
5. As in the Hermitian case treated in \cite{TW4}, the constant $b$ that appears in Theorem \ref{main} is in general a bit mysterious, and cannot be obtained from $\omega$ and $F$ via a simple formula (unless $J$ is integrable and $\de\db\omega=0=\de\db(\omega^2)$, in which case $b=\log\frac{\int_M\omega^n}{\int_M e^F\omega^n}$). One should really think of the pair $(\vp,b)$ as the unique solution of the PDE \eqref{ma}, and in general one cannot expect solutions of PDEs to have simple expressions. What we do know nevertheless is that $|b|\leq \sup_M|F|$, which follows from a simple maximum principle argument.

\medskip

\noindent
{\bf Outline of the proof.} We now discuss the proof of our main results, and relate this to the work of Pli\'s \cite{P} who, as noted above, investigated this problem on $J$-pseudoconvex domains (Harvey-Lawson \cite{HL} obtained weak solutions in the viscosity sense, and in a more general setting).  We prove the \emph{a priori} estimates of Theorem \ref{main2} in four steps: the zero order estimate for $\varphi$, the first order estimate, the real Hessian bound and then higher order estimates.

For the zero order estimate we adapt an approach of Sz\'ekelyhidi \cite{S}, which in turn uses ideas of B{\l}ocki \cite{Bl,Bl2}.  The idea is to  work locally near the infimum of $\varphi$ and use a modification of the Alexandrov-Bakelman-Pucci maximum principle.  Sz\'ekelyhidi's argument holds for a large class of equations, but assuming an  integrable complex structure.  We show that the extra terms arising from the non-integrability can be controlled.  This argument is contained in Section \ref{sectionzeroorder}.

The next step is to bound the first derivatives of $\varphi$, and we carry this out in Section \ref{sectionfirstorder}.  We use a  maximum principle argument, adapted from Pli\'s \cite{P}.   In particular we compute using a unitary frame $e_1, \ldots, e_n$, which turns out to be quite convenient for this problem.  The difference with \cite{P} is that, as would be expected in the compact case, we need to replace Pli\'s's auxiliary global plurisubharmonic function with a certain choice of ``barrier function'' involving $\varphi$ itself.

The heart of the paper is Section \ref{sectionsecondorder}, where we prove a bound on the real Hessian of $\varphi$.  Here, the argument diverges substantially from the study of other similar nonlinear PDEs on complex manifolds where the approach is to bound first the \emph{complex} Hessian of $\varphi$ (see e.g. \cite{Bl3, GL} for instances where a real Hessian bound is obtained for the complex Monge-Amp\`ere equation, with $J$ integrable, after having first obtained a complex Hessian bound).  There are difficulties in carrying out the analogous computation for the complex Hessian here, because of some  linear third order terms which, roughly speaking, involve three ``barred'' or ``unbarred'' partial derivatives of $\varphi$ and cannot be controlled by the usual squares of third order terms.  Instead, we use an important idea of Pli\'s \cite{P}, which is to directly bound the real Hessian of $\varphi$. However, as acknowledged by Pli\'s in a private communication, there is an error in this argument in \cite{P} (in obtaining the first displayed equation of page 981 from the previous lines).
We take a different approach, which corrects \cite{P}, and turns out to be substantially more intricate.  We apply the maximum principle to a quantity involving the largest eigenvalue $\lambda_1$  (cf. \cite{WY, S, STW}) of the real Hessian of $\varphi$.  This gives us some good third order terms which are sufficient, after a series of rather technical lemmas, to push the argument through.   One source of complication is the need to rule out the case when the largest eigenvector of the \emph{real} Hessian of $\varphi$ is in a
direction where the \emph{complex} Hessian is very small (see Lemma \ref{lemmados} below).

Once we have the real Hessian of $\varphi$ bounded, it is then straightforward to obtain the higher order estimates by applying directly an Evans-Krylov type result from \cite{TWWY} and then some standard bootstrapping arguments.  In Section \ref{sectioncomplete} we describe how to obtain the main result Theorem \ref{main} using a continuity argument similar to that of \cite{TW0}.

\bigskip

Before we start with the proofs of the main results,  we describe first some basic results, and notation, on almost complex manifolds.

\bigskip

{\bf Acknowledgments. }We thank S. Pli\'s and G. Sz\'ekelyhidi for some useful communications, and the referees for pertinent comments. The first-named author would like to thank his advisor G. Tian for encouragement and support. The second-named author was partially supported by a Sloan Research Fellowship and NSF grant DMS-1308988, and the third-named author by NSF grant DMS-1406164. This work was carried out while the first-named author was visiting the Department of Mathematics at Northwestern University, supported by the China Scholarship Council (File No. 201506010010), and partly while the second-named author was visiting the Center for Mathematical Sciences and Applications at Harvard University. We would like to thank these institutions for their hospitality and support.

\section{Basic results and notation}\label{sectbas}

Let $M^{2n}$ be a compact manifold of real dimension $2n$ (without boundary) equipped with an almost complex structure $J$, namely an endomorphism of the tangent bundle satisfying $J^2 = - \textrm{Id}$.  The complexified tangent space  $T^{\mathbb{C}}M$ can be decomposed as a direct sum of the two eigenspaces $T^{(1,0)}M$ and $T^{(0,1)}M$ of $J$, corresponding to eigenvalues $\sqrt{-1}$ and $-\sqrt{-1}$ respectively.  Similarly, extending $J$ to $1$-forms $\alpha$ by  $(J\alpha)(X) = - \alpha (JX)$, we obtain  a decomposition of $T^{\mathbb{C}}M^*$ into the $\sqrt{-1}$ and $-\sqrt{-1}$ eigenspaces, spanned by the
 $(0,1)$ and $(1,0)$ forms respectively.  Thus, with the obvious definitions,
 any complex differential form of degree $k$ can be expressed uniquely as a linear combination of $(p,q)$ forms with $p+q=k$.

An \emph{almost Hermitian metric} $g$ on $(M^{2n}, J)$ is a Riemannian metric satisfying
$$g(JX, JY) = g(X,Y), \qquad \textrm{for all } X,Y \in TM.$$
Such metrics always exist in abundance on almost complex manifolds.
The metric $g$ defines a positive definite real $(1,1)$ form $\omega$, given by
$$\omega(X,Y)=g(JX,Y).$$
Conversely, such an $\omega$ defines an almost Hermitian metric $g$ by
$$g(X,Y)=\omega(X,JY).$$
As  we have already done in the introduction, we abuse terminology by referring to the positive definite $(1,1)$ form $\omega$ as an almost Hermitian metric.

The exterior derivative acting on $(p,q)$ forms splits as
$$d=\de+\db+T+\ov{T},$$
where $T$ changes the bidegree by $(2,-1)$ and $\ov{T}$ by $(-1,2)$ (these are essentially given by the Nijenhuis tensor) and so we can define
$$\mn\de\db\vp:=\mn\de(\db\vp).$$
It is immediate to see that this is indeed a real $(1,1)$ form.
A simple calculation shows that
$$\mn\de\db\vp=\frac{1}{2}(d(Jd\vp))^{(1,1)}.$$
We also have (see e.g. \cite[(2.5)]{HL}) that for any two $(1,0)$ vector fields $V,W,$
\begin{equation}\label{hess}
(\de\db \vp)(V,\ov{W})=V\ov{W}(\vp)-[V,\ov{W}]^{(0,1)}(\vp).
\end{equation}

We extend the action of $J$ to $p$-forms by
$$(J\alpha)(X_1,\dots,X_p):=(-1)^p\alpha(JX_1,\dots,JX_p),$$
which satisfies $J^2=(-1)^p$. If $*$ is the Hodge star operator of the Riemannian metric $g$ defined by $\omega$ then the actions of $*$ and $J$ on $p$-forms commute (see e.g. \cite[Lemma 1.10.1]{Gab}). We will say that
$\omega$ is Gauduchon if
\begin{equation}\label{gau}
d^*(Jd^*\omega)=0,
\end{equation}
where $d^*$ is the adjoint of $d$ with respect to $g$.
A short calculation shows that \eqref{gau} is equivalent to
\begin{equation}\label{gau2}
d(Jd(\omega^{n-1}))=0,
\end{equation}
so we see that when $J$ is integrable this condition reduces to the well-known one.
\begin{theorem}[Gauduchon \cite{Ga0}]\label{gaud}
Let $(M^{2n},\omega,J), n\geq 2,$ be a compact almost Hermitian manifold. Then there exists a smooth function $u$, unique up to addition of a constant, such that the conformal almost Hermitian metric $e^u \omega$ is Gauduchon.
\end{theorem}
For convenience, we will give a brief sketch of the proof of this result in the proof of Theorem \ref{solvepoi}.

Let now $(M^{2n},\omega,J)$ be almost Hermitian, and $f$ a smooth function on $M$. We define the canonical Laplacian of $f$ by
\begin{equation} \label{canonical}
\Delta^{C}f=\frac{n\omega^{n-1}\wedge \ddbar f}{\omega^n}=\frac{n\omega^{n-1}\wedge d(Jdf)}{2\omega^n}.
\end{equation}
If $\Delta$ denotes the usual Laplace-Beltrami operator of $g$ then we have the relation
\begin{equation}\label{lapp}
\Delta f = 2\Delta^{C}f +\tau(df),
\end{equation}
where $\tau$ is the ``torsion vector field'' of $(\omega,J)$ (the dual of its Lee form), see e.g. \cite[Lemma 3.2]{T}.

For the ``openness'' part of the continuity method, we will need the following result:
\begin{theorem}[Gauduchon \cite{Ga0}]\label{solvepoi}
Let $(M^{2n},\omega,J), n\geq 2,$ be a compact almost Hermitian manifold.  Fix a nonnegative integer $k$, and $0<\alpha<1$.  Given a function $h \in C^{k, \alpha}(M)$, there exists a function $f \in C^{k+2, \alpha}(M)$ solving
\begin{equation}\label{poi}
\Delta^Cf=h,
\end{equation}
if and only if
\begin{equation}\label{obstr}
\int_M h e^{(n-1)u}\omega^n=0,
\end{equation}
where $u$ is as in Theorem \ref{gaud}. In this case the solution $f$ is unique up to addition of a constant.
\end{theorem}
\begin{proof}
Although this result is not explicitly stated there, the proof is contained in \cite{Ga0}. For the convenience of the reader we give a brief sketch, referring to \cite{Ga0} for details.
Let $\Delta^*$ be the formal $L^2$ adjoint of $\Delta^C$, which is given by
$$\Delta^*f=- \frac{1}{2}d^*(Jd^*(f\omega)),$$
where $f\in C^\infty(X,\mathbb{R})$.
This is an elliptic second order real differential operator, whose index, like the index of $\Delta^C$, is zero. Since the kernel of $\Delta^C$ consists just of constants, it follows that $\dim_{\mathbb{R}}\ker \Delta^*=1$.
The strong maximum principle implies that the only smooth real function in the image of $\Delta^C$ which has constant sign is the zero function, and this implies that every function in the kernel of $\Delta^*$ has constant sign. Thus we can choose a smooth function $f\not\equiv 0$ with $\Delta^*f=0$ and $f\geq 0$, and then another application of the strong maximum principle shows that in fact $f>0$. Therefore we can write $f=e^{(n-1)u}$, for some $u\in C^\infty(X,\mathbb{R})$, and the fact that $\Delta^*f=0$ exactly says that $e^u\omega$ is Gauduchon.

The theorem now follows from the Fredholm alternative, since \eqref{obstr} means that $h$ is orthogonal to $\ker \Delta^*$, which is then equivalent to $h$ being in the image of $\Delta^C$.
\end{proof}

For later use, we have the following result.
\begin{proposition}\label{elle1}
Let $(M^{2n},\omega,J)$ be a compact almost Hermitian manifold, with $\omega$ its associated real $(1,1)$ form. Then there is a constant $C>0$ depending only on $(M, \omega, J)$ such that every smooth function $\vp$ on $M$ which satisfies
\begin{equation}\label{norm}
\omega+\ddbar\vp>0,\quad \sup_M\vp=0,
\end{equation}
also satisfies
\begin{equation}\label{l1}
\int_M (-\vp)\omega^n\leq C.
\end{equation}
\end{proposition}
\begin{proof}
Applying Theorem \ref{gaud} we obtain a smooth function $u$ such that the almost Hermitian metric $\omega'=e^u \omega$ is Gauduchon. Let $\Delta'^C$ be the canonical Laplacian of $\omega'$, which is an elliptic second order differential operator with the kernel consisting of just constants. Standard linear PDE theory (see e.g. \cite[Appendix A]{AS}) shows that there exists a Green function $G$ for $\Delta'^C$ which satisfies
$G(x,y)\geq -C, \|G(x,\cdot)\|_{L^1(M,g)}\leq C,$ for a constant $C>0$, and
$$\vp(x)=\frac{1}{\int_M \omega'^n}\int_M \vp\omega'^n -\int_M \Delta'^C \vp(y) G(x,y)\omega'^n(y),$$
for all smooth functions $\vp$ and all $x\in M$. On the other hand we have
\begin{equation} \label{ibp}
\int_M \Delta'^C \vp\omega'^n=\frac{n}{2}\int_M \omega'^{n-1}\wedge d(Jd\vp)= \frac{n}{2}\int_M (d(Jd\omega'^{n-1})) \varphi=0,
\end{equation}
since $\omega'$ is Gauduchon.  For the ``integration by parts'' step above, we have used the elementary pointwise equality
$$\alpha \wedge (J \beta) = (-1)^p (J \alpha) \wedge \beta,$$
which holds for any $(2n-p)$-form $\alpha$ and $p$-form $\beta$.

Therefore, we are free to add a large uniform constant to $G(x,y)$ to make it nonnegative, while preserving the same Green formula.
If $\vp$ satisfies \eqref{norm} then we have
$$\Delta'^C\vp=e^{-u}\Delta^C\vp\geq -ne^{-u},$$
and so we immediately deduce that $\int_M (-\vp)\omega'^n\leq C,$ from which \eqref{l1} follows.
\end{proof}

We end this section with a remark about the use of the maximum principle.  Consider a function $f \in C^2(M)$.  Then
\begin{equation} \label{localmax}
\ddbar f (x_0) \le  0 \quad \textrm{if $f$ has a local maximum at $x_0$},
\end{equation}
and the reverse inequality holds at a local minimum.  Indeed, the only difference from the integrable case is a first order term which vanishes at $x_0$.

\section{Zero order estimate} \label{sectionzeroorder}

Here we follow the argument given in \cite[Proposition 10]{S} when $J$ is integrable, which is a modification and improvement of an earlier argument by B\l ocki \cite{Bl,Bl2}.

\begin{proposition} \label{uniform estimate}
Let $\varphi$ solve the Monge-Amp\`ere equation (\ref{ma2}).  Then there exists a constant $C$, depending only on $(M,\omega, J)$ and bounds for $F$ such that
$$|\varphi | \le C.$$
\end{proposition}
\begin{proof}
Since $\sup_M\vp=0$, it suffices to derive a uniform lower bound for $I:=\inf_M\vp$. Let $p\in M$ be a point where this infimum is achieved, and choose local coordinates $\{x^1,\dots,x^{2n}\}$ centered at $p$ defined on an open set containing the unit ball $B\subset \mathbb{R}^{2n}$ in its interior. On $B$ let
\begin{equation}\label{quadr}
v(x)=\vp(x)+\ve \sum_{i=1}^{2n}(x^i)^2,
\end{equation}
where $\ve>0$ will be chosen later. Then
$$v(0)=I=\inf_{B}v,$$
$$\inf_{\de B}v\geq v(0)+\ve.$$
We define a set  $P$ by
$$P = \{  x\in B \ | \  |Dv(x)|<\ve/2, \textrm{ and }
v(y)\geq v(x)+Dv(x)\cdot(y-x), \ \forall y \in B\}.$$
Note that $0\in P$, and that $D^2v(x)\geq 0$ as well as $|D\vp(x)|\leq \frac{5}{2}\ve$ for all $x\in P$. Now at any $x\in B$ the symmetric bilinear form $H(v)(X,Y):=(\ddbar v)(X,JY)$ is equal to
\begin{equation}\label{Jhess}
H(v) = \frac{1}{2} (D^2 v)^J + E(v),
\end{equation}
where $(D^2 v)^J$ is the
$J$-invariant part of $D^2v(x)$,
$$(D^2v)^J(x):=\frac{1}{2}(D^2v+J^T\cdot D^2v\cdot J)(x),$$
and $E(v)$ is an error matrix which depends linearly on $Dv(x)$ (see e.g. \cite[p.443]{TWWY}). Therefore, using the fact that $\det (A+B) \ge \det A + \det B$ for symmetric nonnegative definite matrices $A, B$, we have
for all $x\in P$,
\begin{equation}\label{ratta}
\det(D^2v)(x)\leq 2^{2n-1}\det((D^2v)^J)(x)=2^{4n-1}\det(H(v)-E(v))(x).
\end{equation}
(or, one can argue as in \cite{Bl}).  Moreover,
$$(D^2\vp)^J(x) \ge (D^2v)^J(x)-C\ve\ \mathrm{Id}\geq -C\ve\ \mathrm{Id},$$
and using that $|D\vp(x)|\leq 5\ve/2$, together with \eqref{Jhess}, we obtain
$$H(\vp)(x)\geq -C\ve\ \mathrm{Id},$$
for a uniform constant $C$.
Therefore at any $x\in P$ we have
$$\omega+\ddbar\vp\geq\frac{1}{2}\omega,$$
provided $\ve$ is sufficiently small (this fixes the value of $\ve$), but the Monge-Amp\`ere equation \eqref{ma2} then gives us
$$\omega+\ddbar\vp\leq C\omega.$$
From this, using again \eqref{quadr} and $|Dv(x)|<\ve/2$, we conclude that for all $x\in P$ we have
$$0\leq (H(v)-E(v))(x)\leq C\ \mathrm{Id},$$
and \eqref{ratta} gives
$$\det(D^2v)(x)\leq C.$$
Applying the modified Alexandrov-Bakelman-Pucci maximum principle in \cite[Proposition 11]{S} we obtain
$$\ve^{2n}\leq C_{2n}\int_P \det(D^2 v),$$
for a constant $C_{2n}$ which depends only on the dimension of $M$,
and so
$$\ve^{2n}\leq C|P|,$$
where $|P|$ denotes the Lebesgue measure of $P$. But for any $x\in P$ we also have
$$v(x)\leq v(0)+\frac{\ve}{2}=I+\frac{\ve}{2},$$
and we may assume that $I+\frac{\ve}{2}<0$, so
$$|P|\leq \frac{\int_P (-v)}{\left|I+\frac{\ve}{2}\right|}\leq \frac{C}{\left|I+\frac{\ve}{2}\right|},$$
using Proposition \ref{elle1}, and so
$$\left|I+\frac{\ve}{2}\right|\leq \frac{C}{\ve^{2n}},$$
which gives us the desired uniform lower bound for $I$.
\end{proof}

We remark that it is also possible to prove the zero order estimate in our case by the method of Moser iteration \cite{Y1}, much like in \cite{TW1} (where $J$ is integrable), but the calculations are longer. The Moser iteration method was also used by Delan\"oe \cite{De2} for the equation suggested by Gromov (as discussed in the introduction), but in that case the argument is more similar to that of the K\"ahler case.

\section{First order estimate} \label{sectionfirstorder}

In this section, we prove a first order \emph{a priori} estimate for $\varphi$, which uses the zero order estimate of Section \ref{sectionzeroorder}.  This part of the argument is similar to \cite[Lemma 3.3]{P}, except that here we replace Pli\'s's auxiliary plurisubharmonic function by a barrier function involving the solution $\varphi$.

\begin{proposition} \label{gradient estimate}
Let $\varphi$ solve the Monge-Amp\`ere equation (\ref{ma2}).  Then there exists a constant $C$, depending only on $(M,g, J)$ and bounds for $F$ such that
$$| \partial \varphi |_g \le C.$$
\end{proposition}

\begin{proof}
We will prove this estimate by applying the maximum principle to the quantity $Q = e^{f(\varphi)} | \partial \varphi|^2_g$, for a function $f=f(\varphi)$ to be determined later.  We will show that at the maximum point of $Q$, $| \partial \varphi|_g$ is uniformly bounded from above.

First we discuss coordinates.
Let $\{ e_1, \ldots, e_n \}$ be a local frame for $T^{(1,0)}M$ and let $\{ \theta^1, \ldots, \theta^n \}$ be a dual coframe.  We write $g_{i\ov{j}}= g(e_i, \ov{e}_j)$.
The $(1,1)$ form $\omega$ is given by
$$\omega = \sqrt{-1} g_{i\ov{j}} \theta^i \wedge \ov{\theta}^j,$$
where here and henceforth, as should be clear from the context, we are summing over repeated indices (on occasion, for clarity, we will include the summation).
We define
$$\tilde{\omega}  = \omega + \sqrt{-1} \partial \ov{\partial} \varphi,$$
and write $\tilde{g}_{i\ov{j}}$ for the associated metric, defined by  $\tilde{\omega} = \sqrt{-1} \tilde{g}_{i\ov{j}} \theta^i \wedge \ov{\theta}^j$.
Equation \eqref{hess} immediately implies that
$$\sqrt{-1} \partial \bar \partial \varphi = \sqrt{-1}( e_i \ov{e}_j (\varphi) - [e_i, \ov{e}_j]^{(0,1)} (\varphi)) \theta^i \wedge \ov{\theta}^j,$$
and hence
\begin{equation} \label{tildeg}
\tilde{g}_{i\ov{j}} = g_{i\ov{j}} + e_i \ov{e}_j (\varphi) - [e_i, \ov{e}_j]^{(0,1)} (\varphi).
\end{equation}
We now assume for the rest of this section that our local frame $\{ e_1, \ldots, e_n \}$ is unitary with respect to $g$, so that $g_{i\ov{j}}=\delta_{ij}$.
Therefore in these coordinates, $|\partial \varphi|^2_g = \sum_k \varphi_k \varphi_{\ov{k}}$, where we are writing $\varphi_k = e_k(\varphi)$ and $\varphi_{\ov{k}} = \ov{e}_k (\varphi)$ .

Fix a point $x_0$ at which $Q$ achieves its maximum on $M$.  Then, after making a unitary transformation, we may and do assume that $\tilde{g}_{i\ov{j}}$ is diagonal at $x_0$.

Define a second order elliptic operator
$$L := \tilde{g}^{i\ov{j}} (e_i \ov{e}_j - [e_i, \ov{e}_j]^{(0,1)}).$$
Compute at $x_0$, using (\ref{localmax}),
\begin{equation} \label{LQ}
\begin{split}
0 \ge {} &  L(Q)  \\
= {} & \tilde{g}^{i\ov{i}} \big\{ e_i \ov{e}_i (e^{f} | \partial \varphi|^2_g) - [e_i, \ov{e}_i]^{(0,1)} (e^{f} | \partial \varphi|^2_g ) \big\} \\
= {} & \tilde{g}^{i\ov{i}} \big\{ | \partial \varphi|^2_g e_i \ov{e}_i (e^f) + e^f e_i\ov{e}_i (| \partial \varphi|^2_g) + 2 \textrm{Re} (e_i (e^f) \ov{e}_i ( | \partial \varphi|^2_g)) \\
{} & -  | \partial \varphi|^2_g [e_i, \ov{e}_i]^{(0,1)} (e^f) - e^f [e_i, \ov{e}_i ]^{(0,1)} ( | \partial \varphi|^2_g) \big\} \\
= {} & | \partial \varphi|^2_g L(e^f) + e^f L( | \partial \varphi|^2_g) + 2 \textrm{Re} ( \tilde{g}^{i\ov{i}} e_i (e^f) \ov{e}_i (| \partial \varphi|^2_g)).
\end{split}
\end{equation}
We now compute each of these three terms in turn.  First
\begin{equation} \label{Lef}
\begin{split}
L(e^f) = {} & \tilde{g}^{i\ov{i}} \big\{ e_i \ov{e}_i (e^f) - [e_i, \ov{e}_i]^{(0,1)} (e^f) \big\} \\
= {} & \tilde{g}^{i\ov{i}} \big\{ e^f (f')^2 | \varphi_i|^2 + e^f f'' | \varphi_i|^2 + e^f f' e_i \ov{e}_i (\varphi) - e^f f' [e_i, \ov{e}_i]^{(0,1)} (\varphi) \big\} \\
= {} & e^f ( (f')^2 + f'') |\partial \varphi |^2_{\tilde{g}} + n e^f f' - e^f f' \sum_i \tilde{g}^{i\ov{i}},
\end{split}
\end{equation}
where we have used (\ref{tildeg}) for the last line.

Next,
\begin{equation} \label{dps}
\begin{split}
L(| \partial \varphi|^2_g) = {} & \sum_k \tilde{g}^{i\ov{i}} \big\{ e_i \ov{e}_i (\varphi_k \varphi_{\ov{k}} ) - [e_i, \ov{e}_i]^{(0,1)} (\varphi_k \varphi_{\ov{k}}) \big\} \\
= {} & \sum_k \tilde{g}^{i\ov{i}} \big\{ |e_i e_k (\varphi)|^2 + |e_i \ov{e}_k(\varphi)|^2 + \varphi_k e_i \ov{e}_i \ov{e}_k (\varphi) - \varphi_k [e_i, \ov{e}_i]^{(0,1)}\ov{e}_k (\varphi)    \\
 {} & +  \varphi_{\ov{k}}e_i \ov{e}_i e_k (\varphi) - \varphi_{\ov{k}} [e_i, \ov{e}_i]^{(0,1)} e_k (\varphi) \big\}.
\end{split}
\end{equation}
To deal with the terms involving three derivatives of $\varphi$, we use the equation (\ref{ma2}), which we can rewrite in our coordinates as
$$\log \det (\tilde{g}_{i\ov{j}}) = F.$$
Applying $e_k$, we obtain $\tilde{g}^{i\ov{j}} e_k (\tilde{g}_{i\ov{j}}) = F_k$, which at the point $x_0$ gives us
\begin{equation}
\begin{split}
\tilde{g}^{i\ov{i}} (e_k e_i \ov{e}_i (\varphi) - e_k [e_i, \ov{e}_i]^{(0,1)}(\varphi)) = {} & F_k \\
\tilde{g}^{i\ov{i}} (\ov{e}_k e_i \ov{e}_i (\varphi) - \ov{e}_k [e_i, \ov{e}_i]^{(0,1)}(\varphi)) = {} & F_{\ov{k}}.
\end{split}
\end{equation}
Hence
\begin{equation}
\begin{split}
\lefteqn{ \sum_k \tilde{g}^{i\ov{i}} \{ \varphi_k e_i \ov{e}_i \ov{e}_k (\varphi) - \varphi_k [e_i, \ov{e}_i]^{(0,1)}\ov{e}_k (\varphi) \}  } \\
= {} &\sum_k \tilde{g}^{i\ov{i}} \varphi_k \big\{ e_i \ov{e}_k \ov{e}_i(\varphi) - e_i [\ov{e}_k, \ov{e}_i] (\varphi) - \ov{e}_k [e_i, \ov{e}_i]^{(0,1)}(\varphi) + [\ov{e}_k, [e_i, \ov{e}_i]^{(0,1)}](\varphi)     \big\} \\
= {} & \sum_k \tilde{g}^{i\ov{i}} \varphi_k \big\{ \ov{e}_k e_i \ov{e}_i (\varphi) - [\ov{e}_k, e_i] \ov{e}_i (\varphi) - e_i [\ov{e}_k, \ov{e}_i] (\varphi) - \ov{e}_k [e_i, \ov{e}_i]^{(0,1)}(\varphi) \\ {} & + [\ov{e}_k, [e_i, \ov{e}_i]^{(0,1)}] (\varphi) \big\} \\
= {} & \sum_k \varphi_k F_{\ov{k}} + \sum_k \tilde{g}^{i\ov{i}} \varphi_k \big\{ - \ov{e}_i[\ov{e}_k, e_i]  (\varphi) + [\ov{e}_i, [\ov{e}_k, e_i]] (\varphi) - e_i [\ov{e}_k, \ov{e}_i] (\varphi) \\ {} & + [\ov{e}_k, [e_i, \ov{e}_i]^{(0,1)}] (\varphi) \big\} \\
\ge {} & \sum_k \varphi_k F_{\ov{k}} - C | \partial \varphi|_g \sum_k \tilde{g}^{i\ov{i}} ( | \ov{e}_i e_k (\varphi)| + | e_i e_k (\varphi)| ) - C  | \partial \varphi|^2_g\sum_i \tilde{g}^{i\ov{i}},
\end{split}
\end{equation}
where for the last inequality, we have used the fact that  commutators of first order operators, such as $[\ov{e}_k, e_i]$, are first order operators.

Similarly,
\begin{equation}
\begin{split}
\lefteqn{ \sum_k \tilde{g}^{i\ov{i}} \{  \varphi_{\ov{k}}e_i \ov{e}_i e_k (\varphi) - \varphi_{\ov{k}} [e_i, \ov{e}_i]^{(0,1)} e_k (\varphi) \}  } \\
\ge {} & \sum_k \varphi_{\ov{k}} F_k - C | \partial \varphi|_g \sum_k \tilde{g}^{i\ov{i}} ( | \ov{e}_i e_k (\varphi)| + | e_i e_k (\varphi)| ) - C | \partial \varphi|^2_g \sum_i \tilde{g}^{i\ov{i}} .
\end{split}
\end{equation}

Using these last two inequalities in (\ref{dps}), and combining with Young's inequality, we obtain for  $\ve \in (0,1/2]$ (to be determined later),
\begin{equation} \label{dps2}
\begin{split}
L(| \partial \varphi|^2_g)  \ge {} & (1-\ve) \sum_k \tilde{g}^{i\ov{i}} ( |e_i e_k (\varphi)|^2 + |e_i \ov{e}_k (\varphi)|^2 ) \\ {} & -  C \ve^{-1} |\partial \varphi|_g^2 \sum_i \tilde{g}^{i\ov{i}} + 2 \textrm{Re} \left(\sum_k\varphi_k F_{\ov{k}}\right).
\end{split}
\end{equation}

We now deal with the third term on the right hand side of (\ref{LQ}),
\begin{equation} \label{realpart}
\begin{split}
\lefteqn{
2 \textrm{Re} ( \tilde{g}^{i\ov{i}} e_i (e^f) \ov{e}_i (| \partial \varphi|^2_g)) } \\= {} & 2 \textrm{Re} \left(\sum_k  \tilde{g}^{i\ov{i}}  e^f f' \varphi_i  \varphi_{\ov{k}} \ov{e}_i e_k(\varphi)\right) + 2 \textrm{Re} \left(\sum_k  \tilde{g}^{i\ov{i}}  e^f f' \varphi_i  \varphi_k \ov{e}_i \ov{e}_k(\varphi)\right).
\end{split}
\end{equation}
For the first of these terms, we have
\begin{equation} \label{realpart2}
\begin{split}
\lefteqn{
2 \textrm{Re} \left(\sum_k  \tilde{g}^{i\ov{i}}  e^f f' \varphi_i  \varphi_{\ov{k}} \ov{e}_i e_k(\varphi)\right) } \\
= {} & 2 \textrm{Re} \big\{ \sum_k\tilde{g}^{i\ov{i}} e^f f' \varphi_i \varphi_{\ov{k}} ( e_k \ov{e}_i (\varphi) - [e_k, \ov{e}_i]^{(0,1)} (\varphi) - [e_k, \ov{e}_i]^{(1,0)} (\varphi)) \big\} \\
\ge {} & 2 \textrm{Re} \big\{ \sum_k\tilde{g}^{i\ov{i}} e^f f' \varphi_i \varphi_{\ov{k}} (\tilde{g}_{k\bar{i}} - \delta_{ki}) \big\} -   Ce^f |f'| | \partial \varphi|^2_g  \sum_i \tilde{g}^{i\ov{i}}  | \varphi_i| \\
\ge {} & 2 e^f f' | \partial \varphi |^2_g - 2 e^f f' | \partial \varphi |^2_{\ti{g}} - \ve e^f (f')^2 |\partial \varphi|^2_g |\partial \varphi|^2_{\tilde{g}} - C \ve^{-1} e^f | \partial \varphi|^2_g \sum_i \tilde{g}^{i\ov{i}}.
\end{split}
\end{equation}
For the second term of (\ref{realpart}),
\begin{equation} \label{realpart3}
\begin{split}
\lefteqn{2 \textrm{Re} \left(\sum_k  \tilde{g}^{i\ov{i}}  e^f f' \varphi_i  \varphi_k \ov{e}_i \ov{e}_k(\varphi)\right) } \\
\ge {} & - (1-\ve) \sum_k\tilde{g}^{i\ov{i}} e^f |\ov{e}_i \ov{e}_k (\varphi)|^2 - (1+2\ve)e^f (f')^2 |\partial \varphi|^2_g | \partial \varphi|^2_{\tilde{g}},
\end{split}
\end{equation}
using the fact that for any real numbers $a$ and $b$, and $\ve \in (0,1/2]$, $$2ab \ge -(1-\ve)a^2 - (1+2\ve) b^2.$$
Combining (\ref{realpart}), (\ref{realpart2}), (\ref{realpart3}),
\begin{equation} \label{realpartfinal}
\begin{split}
\lefteqn{
2 \textrm{Re} ( \tilde{g}^{i\ov{i}} e_i (e^f) \ov{e}_i (| \partial \varphi|^2_g)) } \\ \ge {} & 2 e^f f' | \partial \varphi |^2_g - 2 e^f f' | \partial \varphi |^2_{\ti{g}} - (1+3\ve) e^f (f')^2 |\partial \varphi|^2_g | \partial \varphi|^2_{\tilde{g}} \\
&   - C \ve^{-1} e^f | \partial \varphi|^2_g \sum_i \tilde{g}^{i\ov{i}}- (1-\ve) \sum_k\tilde{g}^{i\ov{i}} e^f |\ov{e}_i \ov{e}_k (\varphi)|^2 .
\end{split}
\end{equation}

We now put together (\ref{LQ}), (\ref{Lef}), (\ref{dps2}), (\ref{realpartfinal}) to obtain, at $x_0$,
\begin{equation} \label{big}
\begin{split}
0 \ge  {} & e^f (f'' - 3 \ve (f')^2) | \partial \varphi|^2_g |\partial \varphi|^2_{\tilde{g}} +  e^f (-f' - C_0\ve^{-1}) | \partial \varphi|^2_g \sum_i \tilde{g}^{i\ov{i}}\\
{} &  + 2e^f \textrm{Re}\left(\sum_k\varphi_k F_{\ov{k}}\right)  + (2+n) e^f f' | \partial \varphi|^2_g - 2e^f f' | \partial \varphi|^2_{\tilde{g}},
\end{split}
\end{equation}
for a uniform constant $C_0$.  We now choose the function $f=f(\varphi)$ and constant $\ve>0$ as follows.  Define
$$f(\varphi) = \frac{e^{-A(\varphi-1)}}{A}, \quad \ve = \frac{A e^{A(\varphi(x_0)-1)}}{6}.$$
for a large constant $A$ to be determined shortly.  Note that since $\sup_M \varphi=0$, the constant $\ve$ is small.  Compute, at $x_0$,
\[
\begin{split}
f'' - 3 \ve (f')^2 = {} & \frac{Ae^{-A(\varphi(x_0)-1)}}{2} \\
-f' - C_0 \ve^{-1} = {} & \left(1 - \frac{6C_0}{A}\right) e^{-A(\varphi(x_0)-1)}.
\end{split}
\]
In particular, note that $f'$ is \emph{negative}.
Choosing $A= 12C_0$ we obtain lower bounds, at $x_0$,
$$f'' - 3 \ve (f')^2 \ge C^{-1}, \quad -f' - C_0 \ve^{-1} \ge C^{-1},$$
for a uniform constant $C>0$.  In (\ref{big}), after dividing by $e^f$, and increasing $C$ if necessary, we obtain
\[
\begin{split}
0 \ge C^{-1} | \partial \varphi|^2_g |\partial \varphi|^2_{\tilde{g}} + C^{-1} | \partial \varphi|^2_g \sum_i \tilde{g}^{i\ov{i}} - C | \partial \varphi|^2_g - C.
\end{split}
\]
Dividing by $|\partial \varphi|^2_g$ (which we may assume is larger than 1, without loss of generality), we obtain uniform upper bounds for $\sum_i \tilde{g}^{i\ov{i}}$ and $|\partial \varphi|^2_{\tilde{g}}$ at $x_0$.  Combining the upper bound of $\sum_i \tilde{g}^{i\ov{i}}$ with the equation $\det \tilde{g} = e^{F}$, we obtain  a uniform upper bound for $\tilde{g}_{i\ov{i}}$ for each $i$.  Then
$$|\partial \varphi|^2_g =\sum_i | \varphi_i|^2 = \sum_i \tilde{g}_{i\ov{i}} \tilde{g}^{i\ov{i}} | \varphi_i|^2 \le C | \partial \varphi|^2_{\tilde{g}} \le C',$$
as required. This completes the proof of the proposition.
 \end{proof}

Before we end this section, we state the following lemma which we will need later.  It is an immediate consequence of the previous proposition, and the estimate (\ref{dps2}), (but taking now $\ve=1/2$).

\begin{lemma} \label{lemmapv} For a uniform constant $C$,
$$L(| \partial \varphi|^2_g) \ge \frac{1}{2} \sum_k \tilde{g}^{i\ov{i}} (| e_i e_k (\varphi)|^2 + |e_i \ov{e}_k(\varphi)|^2) - C \sum_i \tilde{g}^{i\ov{i}} - C.$$
\end{lemma}

\section{Second order estimate} \label{sectionsecondorder}

In this section we prove:

\begin{proposition}\label{secondord}
Let $\varphi$ solve the Monge-Amp\`ere equation (\ref{ma2}).  Then there exists a constant $C$, depending only on $(M,g, J)$ and bounds for $F$ such that
$$| \nabla^2 \varphi |_g \le C,$$
where $\nabla^2 \varphi$ denotes the real Hessian of $\varphi$ with respect to the metric $g$ (using its Levi-Civita connection).
\end{proposition}
To be precise, the dependence of the constant $C$ on $F$ is as follows: $C$ depends only on upper bounds for $\sup_M F, \sup_M |\de F|_g$, and a lower bound for $\nabla^2 F$ w.r.t. $g$. It does not depend on $\inf_M F$.

\begin{proof}  We first make the preliminary observation that
\begin{equation}\label{stupid}
| \nabla^2 \varphi |_g\leq C\lambda_1( \nabla^2 \varphi)+C',
\end{equation}
everywhere on $M$,
where $\lambda_1( \nabla^2 \varphi)$ is the largest eigenvalue of the real Hessian $\nabla^2 \varphi$ (with respect to the metric $g$). Indeed, if we write $\lambda_1(\nabla^2\vp) \geq \lambda_2(\nabla^2\vp) \ge  \ldots \ge \lambda_{2n}(\nabla^2\vp)$ for all the eigenvalues, then
$$| \nabla^2 \varphi |_g=\left(\sum_{\alpha=1}^{2n}\lambda_\alpha^2\right)^{\frac{1}{2}}\leq C(|\lambda_1|+|\lambda_{2n}|).$$
On the other hand \eqref{lapp} gives
\begin{equation}\label{lapp2}
\sum_{\alpha=1}^{2n}\lambda_\alpha=\Delta\vp=2\Delta^C\vp+\tau(d\vp)\geq -2n+\tau(d\vp)\geq -C,
\end{equation}
using Proposition \ref{gradient estimate}. This inequality implies that $\lambda_1\geq -C,$ and so
$$|\lambda_1|\leq \lambda_1+C,$$
and it also implies that
$$|\lambda_{2n}|\leq C\lambda_1+C,$$
and \eqref{stupid} follows. Therefore it suffices to bound $\lambda_1(\nabla^2\vp)$ from above.
To achieve this we apply the maximum principle to the quantity
$$Q = \log \lambda_1( \nabla^2 \varphi) + h(| \partial \varphi |^2_g) +e^{-A \varphi},$$
on the set $\{ x \in M \ | \ \lambda_1( \nabla^2 \varphi(x)) >0 \}$, which we may assume is nonempty, without loss of generality.
Here $h$ is given by
\begin{equation} \label{defnh}
h(s) = - \frac{1}{2} \log (1+ \sup_M |\partial \varphi|^2_g  - s),
\end{equation}
and $A>1$ is a constant to be determined (which will be uniform, in the sense that it will depend only on the background data).  Observe that $h(| \partial \varphi|^2_g)$ is uniformly bounded, and we have
\begin{equation} \label{proph}
\frac{1}{2}\geq h' \geq \frac{1}{2+2\sup_M |\partial \varphi|^2_g}>0, \quad \textrm{and } h'' = 2 (h')^2,
\end{equation}
where we evaluate $h$ and its derivatives at $|\de\vp|^2_g$.

 Note that $Q$ is a continuous function on its domain, and goes to $-\infty$ on its boundary (if this is nonempty), and hence achieves a maximum at a point $x_0 \in M$ with $\lambda_1(\nabla^2 \varphi(x_0))>0$.  However, $Q$ is not smooth in general, since the eigenspace associated to $\lambda_1$ may have dimension strictly larger than $1$.  We deal with this using a perturbation argument, as in \cite{S, STW}.

First we  discuss the choice of coordinate system in a neighborhood of $x_0$.
Let $V_1$ be a unit vector (with respect to $g$) corresponding to the largest eigenvalue $\lambda_1$ of $\nabla^2 \varphi$, so that at $x_0$,
$$\nabla^2\varphi(V_1,V_1) = \lambda_1.$$
Since $g$ is almost Hermitian, it follows easily that we can find a coordinate system $x^1, x^2, \ldots, x^{2n}$ centered at $x_0$,
such that in the frame $\de_1,\dots,\de_{2n}$ the almost complex structure $J$ is standard at $x_0$ (i.e. $J\de_1=\de_2,$ etc.) and the vectors $\de_1,\dots,\de_{2n}$ are $g$-orthonormal at $x_0$.  Furthermore, after making a quadratic change of coordinates, we  assume that the first derivatives of $g$
vanish at $x_0$:
\begin{equation} \label{fdgv}
\partial_{\gamma} g_{\alpha \beta}|_{x_0} = 0, \quad \textrm{for all } \alpha, \beta, \gamma=1,\dots,2n.
\end{equation}
If at $x_0$ we let
\begin{equation}\label{frame}
e_1 = \frac{1}{\sqrt{2}}(\partial_1 - \sqrt{-1} \partial_2), \ e_2 = \frac{1}{\sqrt{2}}(\partial_3 - \sqrt{-1} \partial_4), \ldots,  e_n = \frac{1}{\sqrt{2}}(\partial_{2n-1} - \sqrt{-1} \partial_{2n}),
\end{equation}
then these form a frame of $(1,0)$ vectors at $x_0$, and we have that $g_{i\ov{j}} := g(e_i, \ov{e}_j)=\delta_{ij}$ (i.e. the frame is $g$-unitary). By performing a further linear change of coordinates at $x_0$, which commutes with $J$, we may assume that
at $x_0$ we have $g_{i\ov{j}}=\delta_{ij}$ and $(\tilde{g}_{i\ov{j}})$ is diagonal with
$$\tilde{g}_{1\ov{1}} \ge \tilde{g}_{2\ov{2}} \ge \cdots \ge \tilde{g}_{n\ov{n}}.$$
This does not affect condition \eqref{fdgv}, or the fact that $J$ is standard at $x_0$, and since $g$ is almost Hermitian we see that the (new) real vectors $\de_1,\dots,\de_{2n}$ are still $g$-orthonormal at $x_0$.

We extend $e_1,\dots,e_n$ smoothly to a $g$-unitary frame of $(1,0)$ vectors in a neighborhood of $x_0$.
The coordinate system and the local unitary frame are now fixed.
Extend $V_1$ to an orthonormal basis $V_1, \ldots, V_{2n}$  of eigenvectors of $\nabla^2\vp$ (with respect to $g$) at $x_0$, corresponding to eigenvalues $\lambda_1(\nabla^2\vp) \geq \lambda_2(\nabla^2\vp) \ge  \ldots \ge \lambda_{2n}(\nabla^2\vp)$. Write $\{ V_{\ \, \beta}^{\alpha} \}_{\alpha=1}^{2n}$ for the components of the vector $V_{\beta}$ at $x_0$, with respect to the coordinates $x^1, \ldots, x^{2n}$ described above.  We extend $V_1, V_2, \ldots, V_{2n}$ to be vector fields in a neighborhood of $x_0$  by taking the components to be constant.  Note that we do not assert that the $V_i$ are eigenvectors for $\nabla^2 \varphi$ outside $x_0$.

We apply a perturbation argument. We define near $x_0$ a smooth section $B= (B_{\alpha \beta})$ of $T^*M \otimes T^*M$ by setting its value in our coordinates $x^1, \ldots, x^{2n}$ at $x_0$ to be
$$B_{\alpha \beta} = \delta_{\alpha \beta} - V_{\ \, 1}^{\alpha} V_{\ \, 1}^{\beta},$$
and extending it to be constant in these coordinates nearby $x_0$.  At $x_0$,
$$B(V_1,V_1) = \sum_{\alpha, \beta =1}^{2n} B_{\alpha \beta} V_{\ \, 1}^{\alpha} V_{\ \, 1}^{\beta} =0,$$
 and $B(Y,Y) = 1$ for any unit vector $Y$ $g$-orthogonal to $V_1$.

 It is convenient to work with endomorphisms of $TM$ instead of symmetric sections of $T^*M \otimes T^*M$.
 We define a local endomorphism $\Phi=( \Phi_{\ \, \beta}^{\alpha})$ of $TM$ by
\begin{equation} \label{definePhi}
\Phi_{\ \, \beta}^{\alpha}=  g^{ \alpha \gamma} \nabla_{\gamma \beta}^2 \varphi -  g^{\alpha \gamma} B_{\gamma \beta}.
\end{equation}
  Now instead of $\lambda_1( \nabla^2 \varphi)$, we consider  $\lambda_1(\Phi)$, the largest eigenvalue of the endomorphism $\Phi$.

Note that $B_{\alpha \beta}$ is nonnegative definite and hence $\lambda_1(\Phi) \le \lambda_1(\nabla^2 \varphi)$ in a neighborhood of $x_0$ while, by definition of $B$ and $V_1$,
$\lambda_1(\Phi) = \lambda_1(\nabla^2 \varphi)$ at $x_0$.  Moreover, at $x_0$, the eigenspace of $\Phi$ corresponding to the largest eigenvalue now has dimension 1, and hence $\lambda_1= \lambda_1(\Phi)$ is smooth in a neighborhood of $x_0$.

We can now consider the perturbed quantity $\hat{Q}$ defined in a neighborhood of $x_0$ by
$$\hat{Q} = \log \lambda_1(\Phi) + h(| \partial \varphi |^2_g) +e^{-A \varphi},$$
which still obtains a local maximum at $x_0$.  $V_1,\dots,V_{2n}$ are eigenvectors for $\Phi$ at $x_0$, corresponding to eigenvalues $\lambda_1(\Phi) > \lambda_2(\Phi) \ge  \ldots \ge \lambda_{2n}(\Phi)$.  In what follows we write $\lambda_{\alpha}$ for $\lambda_{\alpha}(\Phi)$.

The first and second derivatives of $\lambda_1$ at $x_0$ are given by the following lemma.

\begin{lemma}  At $x_0$, we have
\begin{equation} \label{formulae}
\begin{split}
\lambda_1^{\alpha \beta} := {} & \frac{\partial \lambda_1}{\partial \Phi^{\alpha}_{\ \, \beta}} =  V_{\ \, 1}^{\alpha} V_{\ \, 1}^{\beta} \\
\lambda_1^{\alpha \beta, \gamma \delta} := {} & \frac{\partial^2 \lambda_1}{\partial \Phi^{\alpha}_{\ \, \beta} \partial \Phi^{\gamma}_{\ \, \delta}} = \sum_{\mu >1}   \frac{V_{\ \, 1}^{\alpha} V_{\ \, \mu}^{\beta} V_{\ \, \mu}^{\gamma} V_{\ \, 1}^{\delta} + V_{\ \, \mu}^{\alpha} V_{\ \, 1}^{\beta} V_{\ \, 1}^{\gamma} V_{\ \, \mu}^{\delta}}{\lambda_1 - \lambda_{\mu}}, \end{split}
\end{equation}
where Greek indices $\alpha, \beta, \ldots$ go from $1$ to $2n$, unless otherwise indicated.
\end{lemma}
\begin{proof}
Consider the constant orthogonal matrix $V=(V_{\ \, \beta}^{\alpha})$, where $V_{ \ \, \beta}^{\alpha}$ were defined above. It has the property that $V^{T}\Phi V$ is diagonal at $x_{0}$.  Define $\Theta=(\Theta_{\ \, \beta}^{\alpha}):=V^{T}\Phi V$, which has the same eigenvalues as $\Phi$. Since $\Theta$ is diagonal at $x_{0}$, with $\lambda_1$ distinct from all the other eigenvalues, we have the following well-known formulas (see e.g. \cite{Sp}):
\[\begin{split}
\frac{\partial\lambda_{1}(\Theta)}{\partial \Theta_{\ \, \mu}^{\nu}}&=\delta_{1\nu} \delta_{1\mu}\\
\frac{\partial^{2}\lambda_{1}(\Theta)}{\partial \Theta_{\ \, \mu}^{\nu}\partial \Theta_{\ \, \eta}^{\theta}}&=(1-\delta_{1\mu}) \frac{\delta_{1\nu}\delta_{1\eta}\delta_{\mu\theta}}{\lambda_{1}-\lambda_{\mu}}
+(1 - \delta_{1 \eta}) \frac{\delta_{1\theta}\delta_{1\mu}\delta_{\nu\eta}}{\lambda_{1}-\lambda_{\eta}}.
\end{split}\]
Hence, by the chain rule, we compute at $x_{0}$,
\[\begin{split}
\frac{\partial\lambda_{1}(\Phi)}{\partial \Phi_{\ \, \beta}^{\alpha}}=\frac{\partial\lambda_{1}(\Theta)}{\partial \Phi_{\ \, \beta}^{\alpha}}=\frac{\partial\lambda_{1}(\Theta)}{\partial \Theta_{\ \, \mu}^{\nu}}\cdot\frac{\partial \Theta_{\ \, \mu}^{\nu}}{\partial \Phi_{\ \, \beta}^{\alpha}}=\delta_{1\nu} \delta_{1\mu} V_{\ \, \nu}^{\alpha} V_{\ \, \mu}^{\beta}=V_{\ \, 1}^{\alpha}V_{\ \, 1}^{\beta}.
\end{split}\]
Similarly,
\[\begin{split}
\frac{\partial^{2}\lambda_{1}(\Phi)}{\partial \Phi_{\ \, \beta}^{\alpha}\partial \Phi_{\ \, \delta}^{\gamma}}
&=\frac{\partial^{2}\lambda_{1}(\Theta)}{\partial \Theta_{\ \, \mu}^{\nu}\partial \Theta_{\ \, \eta}^{\theta}}\cdot\frac{\partial \Theta_{\ \, \mu}^{\nu}}{\partial \Phi_{\ \, \beta}^{\alpha}}\cdot\frac{\partial \Theta_{\ \, \eta}^{\theta}}{\partial \Phi_{\ \, \delta}^{\gamma}}\\
&=\left((1-\delta_{1\mu})\frac{\delta_{1\nu}\delta_{1\eta}\delta_{\mu\theta}}{\lambda_{1}-\lambda_{\mu}}
+(1-\delta_{1\eta})\frac{\delta_{1\theta}\delta_{1\mu}\delta_{\nu\eta}}{\lambda_{1}-\lambda_{\eta}}\right)V_{\ \, \nu}^{\alpha}V_{\ \, \mu}^{\beta}V_{\ \, \theta}^{\gamma}V_{\ \, \eta}^{\delta}\\
&=(1-\delta_{1\mu})\frac{V_{\ \, 1}^{\alpha}V_{\ \, \mu}^{\beta}V_{\ \, \mu}^{\gamma}V_{\ \, 1}^{\delta}}{\lambda_{1}-\lambda_{\mu}}+(1-\delta_{1\eta})\frac{V_{\ \, \eta}^{\alpha}V_{\ \, 1}^{\beta}V_{\ \, 1}^{\gamma}V_{\ \, \eta}^{\delta}}{\lambda_{1}-\lambda_{\eta}},
\end{split}\]
as required.
\end{proof}

As in the previous section, we denote by $L$ the operator $$L= \tilde{g}^{i\ov{j}} (e_i \ov{e}_j - [e_i, \ov{e}_j]^{(0,1)}).$$  We first prove a lower bound for $L(\lambda_1)$.
In what follows, we assume without loss of generality that $\lambda_1>>1$ at $x_0$.  We also note that from (\ref{ma2}) and the arithmetic-geometric mean inequality,
\begin{equation} \label{ag}
\sum_i\ti{g}^{i\ov{i}} \ge c,
\end{equation}
for a uniform $c>0$.

\begin{lemma} \label{lemmalbl1}  At $x_0$, we have
\begin{equation} \label{l1lb}
\begin{split}
L(\lambda_1) \ge {} & 2 \sum_{\alpha>1} \tilde{g}^{i\ov{i}} \frac{ |e_i (\varphi_{V_{\alpha} V_1})|^2}{\lambda_1 - \lambda_{\alpha}} + \tilde{g}^{p\ov{p}} \tilde{g}^{q\ov{q}} |V_1(\tilde{g}_{p\ov{q}})|^2 \\
& - 2 \tilde{g}^{i\ov{i}} [V_1, e_i] V_1 \ov{e}_i (\varphi) - 2\tilde{g}^{i\ov{i}} [V_1, \ov{e}_i] V_1 e_i (\varphi) - C \lambda_1 \sum_i \tilde{g}^{i\ov{i}},
\end{split}
\end{equation}
where we are writing
$$\vp_{\alpha\beta}=\nabla^2_{\alpha\beta}\vp,\quad \varphi_{V_{\alpha} V_{\beta}} =  \varphi_{\gamma \delta} V_{\ \, \alpha}^{\gamma} V_{\ \, \beta}^{\delta} =\nabla^2\vp(V_\alpha,V_\beta).$$
\end{lemma}
\begin{proof}
Using  (\ref{fdgv}), (\ref{formulae}) and the
fact that $\tilde{g}$ is diagonal at $x_0$,
\begin{equation} \label{Llambda1}
\begin{split}
L(\lambda_1) = {} & \tilde{g}^{i\ov{i}} \lambda_1^{\alpha \beta, \gamma \delta} e_i (\Phi^{\gamma}_{\ \, \delta}) \ov{e}_i (\Phi^{\alpha}_{\ \,  \beta}) + \tilde{g}^{i\ov{i}} \lambda_1^{\alpha \beta} e_i\ov{e}_i (\Phi^{\alpha}_{\ \, \beta}) -\ti{g}^{i\ov{i}}\lambda_1^{\alpha \beta}[e_i, \ov{e}_i]^{(0,1)}(\Phi^{\alpha}_{\ \, \beta})\\
= {} & \tilde{g}^{i\ov{i}} \lambda_1^{\alpha \beta, \gamma \delta} e_i (\varphi_{\gamma \delta}) \ov{e}_i (\varphi_{\alpha \beta}) + \tilde{g}^{i\ov{i}} \lambda_1^{\alpha \beta} e_i\ov{e}_i (\varphi_{\alpha \beta}) + \tilde{g}^{i\ov{i}} \lambda_1^{\alpha \beta} \varphi_{ \gamma \beta} e_i \ov{e}_i (g^{\alpha \gamma}) \\
& -  \tilde{g}^{i\ov{i}} \lambda_1^{\alpha \beta}B_{\gamma \beta} e_i \ov{e}_i (g^{\alpha \gamma})
 -\ti{g}^{i\ov{i}}\lambda_1^{\alpha \beta}[e_i, \ov{e}_i]^{(0,1)}(\vp_{\alpha \beta})\\
\ge {} & 2 \sum_{\alpha >1} \tilde{g}^{i\ov{i}} \frac{ |e_i (\varphi_{V_{\alpha} V_1})|^2}{\lambda_1-\lambda_{\alpha}} + \tilde{g}^{i\ov{i}} e_i \ov{e}_i (\varphi_{V_1 V_1})
-\ti{g}^{i\ov{i}}[e_i, \ov{e}_i]^{(0,1)}(\vp_{V_1V_1}) \\ {} & - C \lambda_1 \sum_i \tilde{g}^{i\ov{i}}.
\end{split}
\end{equation}
Next, we claim that
\begin{equation} \label{blah0}
\begin{split}
\tilde{g}^{i\ov{i}} e_i \ov{e}_i &(\varphi_{V_1 V_1})-\ti{g}^{i\ov{i}}[e_i, \ov{e}_i]^{(0,1)}(\vp_{V_1V_1}) \ge \\
= {} &\tilde{g}^{i\ov{i}} V_1 V_1 (\tilde{g}_{i\ov{i}})  -2\tilde{g}^{i\ov{i}}[V_1,e_i]V_1\ov{e}_i(\vp)-2\tilde{g}^{i\ov{i}}[V_1,\ov{e}_i]V_1 e_i(\vp)- C \lambda_1 \sum_i \tilde{g}^{i\ov{i}}.
\end{split}
\end{equation}
Given this, the lemma follows, since in our coordinates, our equation (\ref{ma2}) is
$$\log \det \tilde{g} = \log \det g + F.$$
Hence, recalling that $g_{i\ov{j}} = g(e_i, \ov{e}_j) =\delta_{ij}$ near $x_0$,
$$\tilde{g}^{i\ov{j}} V_1 (\tilde{g}_{i\ov{j}}) = V_1(F),$$
and applying $V_1$ again,
\begin{equation} \label{blah1}
\begin{split}
\tilde{g}^{i\ov{i}} V_1 V_1 (\tilde{g}_{i\ov{i}}) = {} & \tilde{g}^{p\ov{p}} \tilde{g}^{q\ov{q}} | V_1(\tilde{g}_{p\ov{q}})|^2 + V_1 V_1(F),
\end{split}
\end{equation}
and then \eqref{l1lb} follows from (\ref{Llambda1}), (\ref{blah0}) and (\ref{ag}).

We now give the proof of (\ref{blah0}).  We have
\begin{equation}\label{got0}
\tilde{g}^{i\ov{i}} e_i \ov{e}_i (\varphi_{V_1 V_1})=\tilde{g}^{i\ov{i}} e_i \ov{e}_i  V_1V_1(\vp)-  \tilde{g}^{i\ov{i}} e_i \ov{e}_i (\nabla_{V_1}V_1)(\vp),
\end{equation}
since, for any $\alpha, \beta$,
$$\varphi_{V_{\alpha} V_{\beta}}=V_\alpha V_\beta(\vp)-(\nabla_{V_\alpha}V_\beta)(\vp).$$
First we deal with the second term of (\ref{got0}). We claim that
\begin{equation}\label{got1}
\left|\tilde{g}^{i\ov{i}} e_i \ov{e}_i (\nabla_{V_1}V_1)(\vp)\right|\leq C\lambda_1 \sum_i \ti{g}^{i\ov{i}}.
\end{equation}

For simplicity of notation let $W=\nabla_{V_1}V_1$. Then at $x_0$ we have, using (\ref{tildeg}),
\[\begin{split}
e_i \ov{e}_iW(\vp) = {} & e_i W\ov{e}_i(\vp)+e_i[\ov{e}_i,W](\vp) \\ = {} &
We_i\ov{e}_i(\vp)+[e_i,W]\ov{e}_i(\vp)+e_i[\ov{e}_i,W](\vp)\\
= {} & W(\ti{g}_{i\ov{i}})+W[e_i,\ov{e}_i]^{(0,1)}(\vp)+[e_i,W]\ov{e}_i(\vp)+e_i[\ov{e}_i,W](\vp).
\end{split}\]
Applying $W$ to \eqref{ma2} we have
$$\tilde{g}^{i\ov{i}} W (\tilde{g}_{i\ov{i}}) =   W(F),$$
and \eqref{got1} now follows from (\ref{stupid}) and (\ref{ag}).

Next we note that
\[\begin{split}
-\ti{g}^{i\ov{i}}[e_i, \ov{e}_i]^{(0,1)}(\vp_{V_1V_1})&=-\ti{g}^{i\ov{i}}[e_i, \ov{e}_i]^{(0,1)}V_1V_1(\vp)+\ti{g}^{i\ov{i}}[e_i, \ov{e}_i]^{(0,1)}(\nabla_{V_1}V_1)(\vp)\\
&\geq -\ti{g}^{i\ov{i}}[e_i, \ov{e}_i]^{(0,1)}V_1V_1(\vp)-C\lambda_1\sum_i \ti{g}^{i\ov{i}},
\end{split}\]
for a uniform constant $C$, and so combining this with (\ref{got0}), (\ref{got1}),
\begin{equation}\label{gott1}
\begin{split}
\tilde{g}^{i\ov{i}} e_i \ov{e}_i (\varphi_{V_1 V_1})
&-\ti{g}^{i\ov{i}}[e_i, \ov{e}_i]^{(0,1)}(\vp_{V_1V_1})\\
\geq {} &
\tilde{g}^{i\ov{i}} e_i \ov{e}_i V_1V_1(\vp)
-\ti{g}^{i\ov{i}}[e_i, \ov{e}_i]^{(0,1)}V_1V_1(\vp)
-C\lambda_1\sum_i \ti{g}^{i\ov{i}}.
\end{split}
\end{equation}
In what follows, we write $E$ for a term (which may change from line to line)  which can be bounded by $C\lambda_1\sum_i \tilde{g}^{i\ov{i}}$.  We have
\[
\begin{split}
\tilde{g}^{i\ov{i}} &\big\{ e_i \ov{e}_i V_1 V_1 (\varphi) - [e_i, \ov{e}_i]^{(0,1)} V_1 V_1(\varphi) \big\} \\
= {} & \tilde{g}^{i\ov{i}} \big\{ e_i V_1 \ov{e}_i V_1(\varphi) - e_i [V_1, \ov{e}_i] V_1(\varphi)   - V_1 [e_i, \ov{e}_i]^{(0,1)} V_1(\varphi) \big\} + E\\
= {} & \tilde{g}^{i\ov{i}} \big\{  V_1 e_i \ov{e}_i V_1(\varphi) + [e_i,V_1]\ov{e}_iV_1(\vp)-[V_1,\ov{e}_i]e_iV_1(\vp)  - V_1 V_1 [e_i, \ov{e}_i]^{(0,1)} (\varphi)   \big\} + E\\
= {} & \tilde{g}^{i\ov{i}} \big\{  V_1 e_i V_1 \ov{e}_i (\varphi) - V_1 e_i [V_1, \ov{e}_i] (\varphi) + [e_i,V_1]\ov{e}_iV_1(\vp)-[V_1,\ov{e}_i]e_iV_1(\vp) \\
& - V_1 V_1 [e_i, \ov{e}_i]^{(0,1)} (\varphi) \big\} + E\\
= {} & \tilde{g}^{i\ov{i}} \big\{  V_1 V_1 e_i \ov{e}_i (\varphi) - V_1 [ V_1, e_i] \ov{e}_i (\varphi) - V_1 e_i [V_1, \ov{e}_i] (\varphi) + [e_i,V_1]\ov{e}_iV_1(\vp) \\ & -[V_1,\ov{e}_i]e_iV_1(\vp)
 - V_1 V_1 [e_i, \ov{e}_i]^{(0,1)} (\varphi)  \big\} + E \\
= {} &\tilde{g}^{i\ov{i}} V_1 V_1 \big( e_i\ov{e}_i(\varphi) - [e_i, \ov{e}_i]^{(0,1)}(\varphi)\big)-2\tilde{g}^{i\ov{i}}[V_1,e_i]V_1\ov{e}_i(\vp)-2\tilde{g}^{i\ov{i}}[V_1,\ov{e}_i]V_1 e_i(\vp)+E.
\end{split}
\]
Recalling (\ref{tildeg}), and combining this with \eqref{gott1},  we obtain (\ref{blah0}).
\end{proof}

We use the above lemma to obtain  a lower bound for $L(\hat{Q})$ at $x_0$:

\begin{lemma}  \label{lemmae} For any $\ve$ with $0< \ve \le 1/2$, we have at $x_0$,
\begin{equation} \label{LhatQ4}
\begin{split}
0 \ge {} & L(\hat{Q}) \\ \ge {} & (2-\ve) \sum_{\alpha >1} \tilde{g}^{i\ov{i}} \frac{ |e_i (\varphi_{V_{\alpha} V_1})|^2}{\lambda_1(\lambda_1-\lambda_{\alpha})} + \frac{\tilde{g}^{p\ov{p}} \tilde{g}^{q\ov{q}} | V_1(\tilde{g}_{p\ov{q}})|^2}{\lambda_1}  \\ & {} - (1+\ve)\frac{\tilde{g}^{i\ov{i}} | e_i (\varphi_{V_1 V_1})|^2}{\lambda_1^2}
 + \frac{h'}{2} \sum_k \tilde{g}^{i\ov{i}} (| e_i e_k (\varphi)|^2 + |e_i \ov{e}_k(\varphi)|^2) \\ {} & + h'' \tilde{g}^{i\ov{i}} |e_i | \partial \varphi|^2_g|^2 + \left(Ae^{-A\vp}-\frac{C}{\ve}\right) \sum_i \tilde{g}^{i\ov{i}}  +A^2e^{-A\vp}\tilde{g}^{i\ov{i}} |e_i (\varphi)|^2 \\ & {} - Ane^{-A\vp}.
\end{split}
\end{equation}
\end{lemma}
\begin{proof}
Compute, using Lemma \ref{lemmalbl1}, Lemma \ref{lemmapv} and the inequalities (\ref{proph}) and (\ref{ag}),
\begin{equation} \label{LhatQ2}
\begin{split}
0 \ge {} & L(\hat{Q}) \\ = {} & \frac{L(\lambda_1)}{\lambda_1} - \frac{\tilde{g}^{i\ov{i}} | e_i (\lambda_1)|^2}{\lambda_1^2} + h'L (| \partial \varphi|^2_g) + h'' \tilde{g}^{i\ov{i}} |e_i | \partial \varphi|^2_g|^2 \\ & - Ae^{-A\vp} L(\varphi)
+A^2e^{-A\vp}\tilde{g}^{i\ov{i}} |e_i (\varphi)|^2 \\
\ge {} & 2 \sum_{\alpha >1} \tilde{g}^{i\ov{i}} \frac{ |e_i (\varphi_{V_{\alpha} V_1})|^2}{\lambda_1(\lambda_1-\lambda_{\alpha})} + \frac{\tilde{g}^{p\ov{p}} \tilde{g}^{q\ov{q}} | V_1(\tilde{g}_{p\ov{q}})|^2}{\lambda_1} - \frac{\tilde{g}^{i\ov{i}} | e_i (\varphi_{V_1 V_1})|^2}{\lambda_1^2} \\
{} & -2\frac{\tilde{g}^{i\ov{i}}[V_1,e_i]V_1\ov{e}_i(\vp)+\tilde{g}^{i\ov{i}}[V_1,\ov{e}_i]V_1 e_i(\vp)}{\lambda_1}\\
{} & + \frac{h'}{2} \sum_k \tilde{g}^{i\ov{i}} (| e_i e_k (\varphi)|^2 + |e_i \ov{e}_k(\varphi)|^2) + h'' \tilde{g}^{i\ov{i}} |e_i | \partial \varphi|^2_g|^2 \\ {} &+ (Ae^{-A\vp}-C) \sum_i \tilde{g}^{i\ov{i}}
+A^2e^{-A\vp}\tilde{g}^{i\ov{i}} |e_i (\varphi)|^2
 - Ane^{-A\vp}.
\end{split}
\end{equation}
We deal with the bad third order terms on the second line. At $x_0$ we can write
$$[V_1,e_i]=\sum_{\alpha=1}^{2n}\nu_{i\alpha}V_\alpha,$$
for some complex numbers $\nu_{i\alpha}$ which are uniformly bounded, hence
$$|[V_1,e_i]V_1\ov{e}_i(\vp)+[V_1,\ov{e}_i]V_1 e_i(\vp)|\leq C\sum_{\alpha=1}^{2n}|V_\alpha V_1e_i(\vp)|,$$
and
\[\begin{split}
V_\alpha V_1e_i(\vp)&=e_iV_\alpha V_1(\vp)+V_\alpha [V_1,e_i](\vp)+[V_\alpha,e_i]V_1(\vp)\\
&=e_i(\vp_{V_\alpha V_1})+e_i(\nabla_{V_\alpha}V_1)(\vp)+V_\alpha [V_1,e_i](\vp)+[V_\alpha,e_i]V_1(\vp),
\end{split}\]
and so
\[\begin{split}
 2&\frac{\tilde{g}^{i\ov{i}}[V_1,e_i]V_1\ov{e}_i(\vp)+\tilde{g}^{i\ov{i}}[V_1,\ov{e}_i]V_1 e_i(\vp)}{\lambda_1} \\
&\leq C\frac{\ti{g}^{i\ov{i}}|e_i(\vp_{V_1V_1})|}{\lambda_1}
+C\sum_{\alpha>1}\frac{\ti{g}^{i\ov{i}}|e_i(\vp_{V_\alpha V_1})|}{\lambda_1}+C\sum_i \ti{g}^{i\ov{i}}\\
&\leq \ve\frac{\tilde{g}^{i\ov{i}} | e_i (\varphi_{V_1 V_1})|^2}{\lambda_1^2}+ \ve \sum_{\alpha >1} \tilde{g}^{i\ov{i}} \frac{ |e_i (\varphi_{V_{\alpha} V_1})|^2}{\lambda_1(\lambda_1-\lambda_{\alpha})}+\frac{C}{\ve}\sum_i \ti{g}^{i\ov{i}}+\frac{C}{\ve}\sum_i\ti{g}^{i\ov{i}}\sum_{\alpha>1}\frac{\lambda_1-\lambda_\alpha}{\lambda_1}\\
&\leq \ve\frac{\tilde{g}^{i\ov{i}} | e_i (\varphi_{V_1 V_1})|^2}{\lambda_1^2}+ \ve \sum_{\alpha >1} \tilde{g}^{i\ov{i}} \frac{ |e_i (\varphi_{V_{\alpha} V_1})|^2}{\lambda_1(\lambda_1-\lambda_{\alpha})}+\frac{C}{\ve}\sum_i \ti{g}^{i\ov{i}},
\end{split}\]
where we used \eqref{lapp2}. Thus the lemma follows from (\ref{LhatQ2}).
\end{proof}

We need to deal with the ``bad'' negative term
\begin{equation} \label{badterm}
- (1+\ve)\frac{\tilde{g}^{i\ov{i}} | e_i (\varphi_{V_1 V_1})|^2}{\lambda_1^2}
\end{equation}
in (\ref{LhatQ4}).
We split up into different cases.  The constant $\ve>0$  will be different in
  each case.

\bigskip
\noindent
{\bf Case 1.} \ Assume that, at $x_0$,
\begin{enumerate}
\item[(a)] $\ti{g}_{1\ov{1}} < A^3e^{-2A\vp}\ti{g}_{n\ov{n}}$, \ {\bf or}

\vspace{3pt}

\item[(b)] we have
$$ \frac{h'}{4} \sum_k \tilde{g}^{i\ov{i}} (| e_i e_k (\varphi)|^2 + |e_i \ov{e}_k(\varphi)|^2) > 6 (\sup_M | \partial \varphi|^2_g )A^2 e^{-2A\vp}\sum_i \tilde{g}^{i\ov{i}}.$$
\end{enumerate}

\bigskip

In this case we just choose $\ve=\frac{1}{2}$.
We use the fact that, at $x_0$, the first derivative of $\hat{Q}$ vanishes.  Hence, using the elementary inequality $|a+b|^2 \le 4|a|^2 + \frac{4}{3}|b|^2$, for $a, b \in \mathbb{C}$,
\begin{equation} \label{3o2}
\begin{split}
- \frac{3}{2}\frac{\tilde{g}^{i\ov{i}} |e_i (\varphi_{V_1 V_1})|^2}{\lambda_1^2} = {} & - \frac{3}{2}\tilde{g}^{i\ov{i}} | Ae^{-A\vp} e_i(\varphi) - h' e_i (| \partial \varphi|^2_g)|^2 \\
\ge {} & - 6 (\sup_M | \partial \varphi|^2_g )A^2e^{-2A\vp} \sum_i \tilde{g}^{i\ov{i}} \\ &  - 2 (h')^2 \tilde{g}^{i\ov{i}} |e_i | \partial \varphi|^2_g|^2.
\end{split}
\end{equation}

Hence from (\ref{LhatQ4}), discarding some positive terms, we obtain at $x_0$,
\[
\begin{split}
0 \ge {} &   - 6 (\sup_M | \partial \varphi|^2_g )A^2 e^{-2A\vp}\sum_i \tilde{g}^{i\ov{i}} + (h'' - 2 (h')^2) \tilde{g}^{i\ov{i}} |e_i | \partial \varphi|^2_g|^2 \\
{} & + \frac{h'}{2} \sum_k \tilde{g}^{i\ov{i}} (| e_i e_k (\varphi)|^2 + |e_i \ov{e}_k(\varphi)|^2)
 + (Ae^{-A\vp}-C) \sum_i \tilde{g}^{i\ov{i}} - Ane^{-A\vp}.
\end{split}
\]
But from (\ref{proph}), we have $h'' = 2 (h')^2$.  In case (a) note that all the numbers $\ti{g}_{i\ov{i}}, 1\leq i\leq n,$ are comparable to each other up to a uniform constant (which depends on $A$). Using again \eqref{ag}, as well as Propositions \ref{uniform estimate} and \ref{gradient estimate}, we have
$$0 \ge \frac{h'}{2} \sum_k \tilde{g}^{i\ov{i}} (| e_i e_k (\varphi)|^2 + |e_i \ov{e}_k(\varphi)|^2) - C_A\sum_i \ti{g}^{i\ov{i}},$$
for a uniform constant $C_A$ (depending on the uniform constant $A$), and since all the $\ti{g}^{i\ov{i}}$ are comparable to each other this gives
\begin{equation}\label{nabb}
\sum_{i,k} (| e_i e_k (\varphi)|^2 + |e_i \ov{e}_k(\varphi)|^2)\leq C_A.
\end{equation}
But at $x_0$ we can define the complex covariant derivatives $\vp_{e_i \ov{e}_k}$ and $\vp_{e_i e_k}$ in the obvious way, which satisfy
$$\varphi_{e_i \ov{e}_k}-e_i\ov{e}_k(\vp)=-(\nabla_{e_i}\ov{e}_k)(\vp),\quad \varphi_{e_i e_k}-e_i e_k(\vp)=-(\nabla_{e_i} e_k)(\vp),$$
and these differences are uniformly bounded thanks to Proposition \ref{gradient estimate}.
Therefore
\begin{equation}\label{nabb3}
\sum_{i,k} (| \vp_{e_i e_k}|^2 + |\vp_{e_i \ov{e}_k}|^2)\leq C_A,
\end{equation}
and recalling \eqref{frame}, we see that
\begin{equation}\label{nabb2}
\sum_{\alpha,\beta}|\nabla^2_{\alpha\beta}\vp|\leq C_A,
\end{equation}
and so $\lambda_1(x_0)$ is uniformly bounded.
Thanks to Propositions \ref{uniform estimate} and \ref{gradient estimate},
this shows that $Q$ is bounded from above at $x_0$, and hence everywhere.

On the other hand, in case (b), we obtain
\begin{equation} \label{A2}
0 \ge \frac{h'}{4} \sum_k \tilde{g}^{i\ov{i}} (| e_i e_k (\varphi)|^2+  |e_i \ov{e}_k(\varphi)|^2) + (A-C_0) \sum_i \tilde{g}^{i\ov{i}} - Ane^{-A\vp},
\end{equation}
for a uniform $C_0$.  Then as long as
\begin{equation} \label{A1}
A \ge C_0+1,
\end{equation}
we see that at $x_0$ we have $\sum_i\tilde{g}^{i\ov{i}}\leq C$, and so also
\begin{equation}\label{dumb}
\tr{g}{\ti{g}}\leq (\tr{\ti{g}}{g})^{n-1}\frac{\ti{\omega}^n}{\omega^n}\leq C.
\end{equation}
Therefore at $x_0$ we have that $\ti{g}$ is uniformly equivalent to $g$, and \eqref{A2} then shows that \eqref{nabb} holds again, and so $Q$ is bounded from above at $x_0$.  This completes Case 1.

\bigskip
\noindent
{\bf Case 2.} Neither (a) nor (b) of Case 1 hold.
\bigskip

This is the difficult case, and it will take up the rest of this section.  Let
$$I=\{i\in \{1,\dots,n\}\ |\ \ti{g}_{i\ov{i}}(x_0)\geq A^3e^{-2A\vp(x_0)}\ti{g}_{n\ov{n}}(x_0)\}.$$
We have $n\not\in I$ (since $A>1$), and $1\in I$ (otherwise we are  in part (a) of Case 1).
So for example when $n=2$ we have $I=\{1\}$.

We will use four lemmas to complete the proof.  The first deals with the relatively harmless part of the bad term (\ref{badterm}).

\begin{lemma} \label{lemmauno} Assuming that $A$ is larger than $6n\sup_M | \partial \varphi|^2_g$, we have
\begin{equation}\label{done0}
-(1+\ve)\sum_{i\in I}\frac{\tilde{g}^{i\ov{i}} | e_i (\varphi_{V_1 V_1})|^2}{\lambda_1^2}
 \geq -\sum_i \ti{g}^{i\ov{i}}- 2 (h')^2 \sum_{i\in I}\tilde{g}^{i\ov{i}} |e_i | \partial \varphi|^2_g|^2.
\end{equation}
\end{lemma}
\begin{proof} Using the fact that the first derivative of $\hat{Q}$ vanishes at $x_0$, as in (\ref{3o2}), we compute
\begin{equation}
\begin{split}
\lefteqn{-(1+\ve)\sum_{i\in I}\frac{\tilde{g}^{i\ov{i}} | e_i (\varphi_{V_1 V_1})|^2}{\lambda_1^2} } \\&=-(1+\ve) \sum_{i\in I}\tilde{g}^{i\ov{i}} | A e^{-A\vp}e_i(\varphi) - h' e_i (| \partial \varphi|^2_g)|^2 \\
 &\ge - 6 (\sup_M | \partial \varphi|^2_g )A^2e^{-2A\vp} \sum_{i\in I} \tilde{g}^{i\ov{i}}  - 2 (h')^2 \sum_{i\in I}\tilde{g}^{i\ov{i}} |e_i | \partial \varphi|^2_g|^2\\
 &\geq -\frac{6n (\sup_M | \partial \varphi|^2_g ) \ti{g}^{n\ov{n}}}{A}- 2 (h')^2 \sum_{i\in I}\tilde{g}^{i\ov{i}} |e_i | \partial \varphi|^2_g|^2\\
 &\geq -\sum_i \ti{g}^{i\ov{i}}- 2 (h')^2 \sum_{i\in I}\tilde{g}^{i\ov{i}} |e_i | \partial \varphi|^2_g|^2,
\end{split}
\end{equation}
where in the last line we used the assumption $A \ge 6n\sup_M | \partial \varphi|^2_g$.
\end{proof}

The next lemma shows that, roughly speaking, and modulo terms of order $O(\lambda_1^{-1})$, the largest eigenvector $V_1$ lies in the directions where the Hermitian metric $(\tilde{g}_{i\ov{i}})$ at $x_0$ is not too small.  More precisely, we define in our coordinate patch a $(1,0)$ vector field by $$\ti{e}_1:=\frac{1}{\sqrt{2}}(V_1-\mn JV_1).$$
We write at $x_0$,
\begin{equation} \label{fix}
\ti{e}_1=\sum_{q=1}^n \nu_qe_q,\quad \sum_{q=1}^n|\nu_q|^2=1,
\end{equation}
for complex numbers $\nu_1, \ldots, \nu_n$, where the second equation follows from the fact that $\ti{e}_1$ is $g$-unit at $x_0$.
Then we have:

\begin{lemma} \label{lemmados}  For a uniform constant $C_A$ depending on $A$, we have
$$| \nu_q | \le \frac{C_A}{\lambda_1}, \quad \textrm{for all } q \notin I,$$
where the $\nu_q$ are given by (\ref{fix}).
\end{lemma}
\begin{proof}
To prove the claim, we see from
$$ \frac{h'}{4} \sum_k \tilde{g}^{i\ov{i}} (| e_i e_k (\varphi)|^2 + |e_i \ov{e}_k(\varphi)|^2) \le 6 (\sup_M | \partial \varphi|^2_g )A^2e^{-2A\vp} \sum_i \tilde{g}^{i\ov{i}},$$
that
$$ \frac{h'}{4} \sum_k \sum_{i \notin I} \tilde{g}^{i\ov{i}} (| e_i e_k (\varphi)|^2 + |e_i \ov{e}_k(\varphi)|^2) \le 6n  (\sup_M | \partial \varphi|^2_g )A^2e^{-2A\vp}  \tilde{g}^{n\ov{n}}.$$
Hence by the definition of $I$,
\begin{equation} \label{sumki}
\sum_k \sum_{i \notin I} (| e_i e_k (\varphi)|^2 + |e_i \ov{e}_k(\varphi)|^2) \le  \frac{24n A^5 e^{-4A  \varphi}}{h'} \sup_M | \partial \varphi|^2_g.
\end{equation}
For convenience, write $I = \{ 1, 2, \ldots, j \}$ for some $j$ with $1 \leq j < n$.  Then, arguing exactly as in \eqref{nabb}, \eqref{nabb3}, \eqref{nabb2}, the estimate (\ref{sumki})  implies that
\begin{equation} \label{nablas}
\sum_{\alpha = 2j+1}^{2n} \sum_{\beta = 1}^{2n} | \nabla^2_{\alpha \beta} \varphi | \le C_A,
\end{equation}
where we write $C_A$ for a uniform constant depending on $A$ (recall that $A$ is yet to be determined, but will be chosen uniformly).

In terms of the matrix $\Phi^\alpha_{\ \, \beta}$, defined by (\ref{definePhi}), the inequality (\ref{nablas}) implies that at $x_0$,
 $$| \Phi^{\alpha}_{\ \, \beta} | \le C_A, \qquad \textrm{whenever } 2j+1 \le \alpha \le 2n, \ 1 \le \beta \le 2n.$$
Since $V_1 = ( V_{\ \, 1}^{\alpha})_{\alpha=1}^{2n}$ is a unit eigenvector for $\Phi$ with eigenvalue $\lambda_1$ we have $\Phi (V_1) = \lambda_1 V_1$ and hence
$$|V_{ \ \, 1}^{\alpha}| = \left| \frac{1}{\lambda_1} \sum_{\beta=1}^{2n} \Phi^{\alpha}_{\ \, \beta}V_{\ \, 1}^{\beta} \right| \le \frac{C_A}{\lambda_1}, \quad \textrm{for } 2j+1 \le \alpha \le 2n.$$
The lemma then follows easily from the definition of the $e_i$.
\end{proof}

The goal for our last two lemmas is to obtain a lower bound for the first three terms on the right hand side of (\ref{LhatQ4}), excluding those bounded by Lemma \ref{lemmauno}.   The first of these two lemmas is a technical intermediate step.  We make use of real numbers $\mu_2, \mu_3, \ldots, \mu_{2n}$ defined by
\begin{equation} \label{defnmu}
JV_1=\sum_{\alpha>1} \mu_\alpha V_\alpha,\quad \sum_{\alpha>1}\mu_\alpha^2=1, \quad \textrm{at } x_0,
\end{equation}
where we recall that at $x_0$ the vector $JV_1$ is $g$-unit and $g$-orthogonal to $V_1$.

\begin{lemma} \label{lemmatres} For any $\gamma>0$,   at $x_0$,
\[
\begin{split}
(2-\ve) \sum_{\alpha >1} &\tilde{g}^{i\ov{i}} \frac{ |e_i (\varphi_{V_{\alpha} V_1})|^2}{\lambda_1(\lambda_1-\lambda_{\alpha})} + \frac{\tilde{g}^{p\ov{p}} \tilde{g}^{q\ov{q}} | V_1(\tilde{g}_{p\ov{q}})|^2}{\lambda_1} -(1+\ve)\sum_{i\not\in I}\frac{\tilde{g}^{i\ov{i}} | e_i (\varphi_{V_1 V_1})|^2}{\lambda_1^2}\\
&\geq
\sum_{i\not\in I}\sum_{\alpha >1} \frac{\tilde{g}^{i\ov{i}}}{\lambda_1^2}\left(\frac{(2-\ve)\lambda_1}{\lambda_1-\lambda_{\alpha}} |e_i (\varphi_{V_{\alpha} V_1})|^2\right)
+2\sum_{q\in I}\sum_{i\not\in I}\tilde{g}^{i\ov{i}}\tilde{g}^{q\ov{q}}\frac{| V_1(\tilde{g}_{i\ov{q}})|^2}{\lambda_1}\\
&-3\ve \sum_{i\not\in I}\frac{\tilde{g}^{i\ov{i}} | e_i (\varphi_{V_1 V_1})|^2}{\lambda_1^2}
-(1-\ve)(1+\gamma)\ti{g}_{\ti{1}\ov{\ti{1}}}\sum_{q\in I}\sum_{i\not\in I}\frac{2\tilde{g}^{i\ov{i}}\ti{g}^{q\ov{q}}}{\lambda_1^2} |V_1(\ti{g}_{i\ov{q}})|^2\\
&-\frac{C}{\ve}\sum_{i}\tilde{g}^{i\ov{i}}-(1-\ve)\left(1+\frac{1}{\gamma}\right)\left(\lambda_1-\sum_{\alpha>1}\lambda_\alpha\mu_\alpha^2\right)
\sum_{i\not\in I}\sum_{\alpha>1}\frac{\tilde{g}^{i\ov{i}}}{\lambda_1^2}\frac{|e_i(\vp_{V_\alpha V_1})|^2}{\lambda_1-\lambda_\alpha},
\end{split}
\]
assuming without loss of generality that at $x_0$,  $\lambda_1 \ge C_A/\ve$ for a uniform $C_A>0$ depending on $A$.  Here we are writing $\ti{g}_{\ti{1}\ov{\ti{1}}}:=\ti{g}(\ti{e}_1,\ov{\ti{e}}_1)$.
\end{lemma}
\begin{proof}
We first claim that
\begin{equation} \label{claimOl1}
e_i (\varphi_{V_1 V_1})=\sqrt{2}\sum_{q}\ov{\nu_q}V_1(\ti{g}_{i\ov{q}})-\mn\sum_{\alpha>1}\mu_\alpha e_i(\vp_{V_1V_\alpha}) + E,
\end{equation}
where $E$ denotes a term satisfying  $|E| \le C \lambda_1$ for a uniform $C$.
Defining $\vp_{V_1\ov{\ti{e}}_1}$ in the obvious way, we have at $x_0$,
\begin{equation}\label{p1}
e_i (\varphi_{V_1 V_1})=\sqrt{2}e_i(\vp_{V_1\ov{\ti{e}}_1}) - \mn e_i (\varphi_{V_1 JV_1}),
\end{equation}
since $\ov{\tilde{e}}_1 = \frac{1}{\sqrt{2}} (V_1+ \sqrt{-1} JV_1)$.  Next,
\begin{equation}\label{p2}\begin{split}
e_i(\vp_{V_1\ov{\ti{e}}_1})&=e_iV_1\ov{\ti{e}}_1 (\vp)-e_i(\nabla_{V_1}\ov{\ti{e}}_1)(\vp)
=\ov{\ti{e}}_1e_iV_1(\vp)+ E,
\end{split}
\end{equation}
since the error terms arising from switching the order of operators give terms involving only two derivatives of $\varphi$.

With $\nu_1, \ldots, \nu_n \in \mathbb{C}$ as in (\ref{fix}), we have at $x_0$,
\begin{equation}\label{p3}\begin{split}
\ov{\ti{e}}_1e_iV_1(\vp)= {} &\sum_{q}\ov{\nu_q} \, \ov{e}_qe_iV_1(\vp)
=  \sum_{q}\ov{\nu_q}  V_1e_i\ov{e}_q(\vp) + E  \\
= {} &\sum_{q}\ov{\nu_q} V_1(\ti{g}_{i\ov{q}})+ E.
\end{split}\end{equation}
Recalling (\ref{defnmu}),
\begin{equation}\label{p4}
\begin{split}
e_i(\vp_{V_1JV_1})= {} & e_i V_1 JV_1(\vp)-e_i(\nabla_{V_1}(JV_1))(\vp)
=JV_1 e_iV_1(\vp)+E \\
= {} &\sum_{\alpha>1}\mu_\alpha V_\alpha e_iV_1(\vp)+E
=  \sum_{\alpha>1}\mu_\alpha e_i(\vp_{V_1V_\alpha})+E.
\end{split}
\end{equation}
Combining \eqref{p1}, \eqref{p2}, \eqref{p3} and \eqref{p4} proves the claim (\ref{claimOl1}).

Hence, for a uniform $C$,
\begin{equation}\label{aux1}
\begin{split}
(1+\ve)&\sum_{i\not\in I}\frac{\tilde{g}^{i\ov{i}} | e_i (\varphi_{V_1 V_1})|^2}{\lambda_1^2}=3\ve\sum_{i\not\in I}\frac{\tilde{g}^{i\ov{i}} | e_i (\varphi_{V_1 V_1})|^2}{\lambda_1^2}\\
&+(1-2\ve)\sum_{i\not\in I}\frac{\tilde{g}^{i\ov{i}} | \sqrt{2}\sum_{q=1}^n\ov{\nu_q}V_1(\ti{g}_{i\ov{q}}) - \mn \sum_{\alpha>1}\mu_\alpha e_i(\vp_{V_1V_\alpha})+E|^2}{\lambda_1^2}\\
&\leq 3\ve\sum_{i\not\in I}\frac{\tilde{g}^{i\ov{i}} | e_i (\varphi_{V_1 V_1})|^2}{\lambda_1^2}\\
&+(1-\ve)\sum_{i\not\in I}\frac{\tilde{g}^{i\ov{i}} | \sqrt{2}\sum_{q\in I}\ov{\nu_q}V_1(\ti{g}_{i\ov{q}}) - \mn \sum_{\alpha>1}\mu_\alpha e_i(\vp_{V_1V_\alpha})|^2}{\lambda_1^2} \\
& + \frac{C_A}{\ve} \sum_{i \notin I} \sum_{q \notin I} \frac{\tilde{g}^{i\ov{i}}|V_1(\tilde{g}_{i\ov{q}})|^2}{\lambda_1^4}    +\frac{C}{\ve}\sum_{i}\tilde{g}^{i\ov{i}},
\end{split}
\end{equation}
using Lemma \ref{lemmados}.

For  $\gamma>0$, we bound
\begin{equation}\label{bad}
\begin{split}
&(1-\ve)\sum_{i\not\in I}\frac{\tilde{g}^{i\ov{i}} | \sqrt{2}\sum_{q\in I}\ov{\nu_q}V_1(\ti{g}_{i\ov{q}}) - \mn \sum_{\alpha>1}\mu_\alpha e_i(\vp_{V_1V_\alpha})|^2}{\lambda_1^2}\\
&\leq(1-\ve)(1+\gamma)\sum_{i\not\in I}\frac{2\tilde{g}^{i\ov{i}}}{\lambda_1^2} \left|\sum_{q\in I}\ov{\nu_q}V_1(\ti{g}_{i\ov{q}})\right|^2+(1-\ve)\left(1+\frac{1}{\gamma}\right)\sum_{i\not\in I}\frac{\tilde{g}^{i\ov{i}}}{\lambda_1^2} \left|\sum_{\alpha>1}\mu_\alpha e_i(\vp_{V_1V_\alpha})\right|^2.
\end{split}
\end{equation}
But we also have, by the Cauchy-Schwarz inequality,
$$\left|\sum_{q\in I}\ov{\nu_q}V_1(\ti{g}_{i\ov{q}})\right|^2\leq \left(\sum_{q =1}^n|\nu_q|^2\ti{g}_{q\ov{q}}\right) \left(\sum_{q\in I}\ti{g}^{q\ov{q}}|V_1(\ti{g}_{i\ov{q}})|^2\right)
=\ti{g}(\ti{e}_1,\ov{\ti{e}}_1)\left(\sum_{q\in I}\ti{g}^{q\ov{q}}|V_1(\ti{g}_{i\ov{q}})|^2\right),$$
and
\[\begin{split}\left|\sum_{\alpha>1}\mu_\alpha e_i(\vp_{V_\alpha V_1})\right|^2&\leq\left(\sum_{\alpha>1}(\lambda_1-\lambda_\alpha)\mu_\alpha^2\right)\left(\sum_{\alpha>1}\frac{|e_i(\vp_{V_\alpha V_1})|^2}{\lambda_1-\lambda_\alpha}\right)\\
&=\left(\lambda_1-\sum_{\alpha>1}\lambda_\alpha\mu_\alpha^2\right)\left(\sum_{\alpha>1}\frac{|e_i(\vp_{V_\alpha V_1})|^2}{\lambda_1-\lambda_\alpha}\right),
\end{split}\]
recalling that we have $\sum_{\alpha>1}\mu_\alpha^2=1$. So the conclusion is that
\begin{equation}\label{bad2}
\begin{split}
&(1-\ve)\sum_{i\not\in I}\frac{\tilde{g}^{i\ov{i}} | \sqrt{2}\sum_{q\in I}\ov{\nu_q}V_1(\ti{g}_{i\ov{q}}) - \mn \sum_{\alpha>1}\mu_\alpha e_i(\vp_{V_1V_\alpha})|^2}{\lambda_1^2}\\
&\leq(1-\ve)(1+\gamma)\ti{g}_{\ti{1}\ov{\ti{1}}}\sum_{q\in I}\sum_{i\not\in I}\frac{2\tilde{g}^{i\ov{i}}\ti{g}^{q\ov{q}}}{\lambda_1^2} |V_1(\ti{g}_{i\ov{q}})|^2\\
&+(1-\ve)\left(1+\frac{1}{\gamma}\right)\left(\lambda_1-\sum_{\alpha>1}\lambda_\alpha\mu_\alpha^2\right)
\sum_{i\not\in I}\sum_{\alpha>1}\frac{\tilde{g}^{i\ov{i}}}{\lambda_1^2}\frac{|e_i(\vp_{V_\alpha V_1})|^2}{\lambda_1-\lambda_\alpha}.
\end{split}
\end{equation}
Also, using \eqref{tildeg}, we see that at $x_0$ we have
$$\ti{g}_{q\ov{q}}\leq \ti{g}_{1\ov{1}}\leq C+e_1\ov{e}_1(\vp)\leq C'\lambda_1,$$
for any $q$, using that $\lambda_1$ may be assumed to be large. Therefore, if $\lambda_1\geq \frac{C_A}{\ve}$ for the constant $C_A$ of (\ref{aux1}), and if also $\lambda_1\geq C'$, then we have
$$\frac{C_A}{\ve \lambda_1^3} \le \tilde{g}^{q\ov{q}}.$$
Therefore we obtain the inequality
\begin{equation}\label{good}
\begin{split}
\frac{\tilde{g}^{p\ov{p}} \tilde{g}^{q\ov{q}} | V_1(\tilde{g}_{p\ov{q}})|^2}{\lambda_1}\geq {} &  2\sum_{q\in I}\sum_{i\not\in I}\tilde{g}^{i\ov{i}}\tilde{g}^{q\ov{q}}\frac{| V_1(\tilde{g}_{i\ov{q}})|^2}{\lambda_1}  \\ {} &+ \frac{C_A}{\ve} \sum_{q \notin I}\sum_{i \notin I}  \frac{\tilde{g}^{i\ov{i}}|V_1(\tilde{g}_{i\ov{q}})|^2}{\lambda_1^4}.
\end{split}
\end{equation}
Combining \eqref{aux1}, \eqref{bad2} and \eqref{good} completes the proof of the lemma.
\end{proof}

We use this to prove the final lemma:
\begin{lemma}\label{lemmacuatro}
At $x_0$ we have
\[
\begin{split}
(2-\ve) \sum_{\alpha >1} &\tilde{g}^{i\ov{i}} \frac{ |e_i (\varphi_{V_{\alpha} V_1})|^2}{\lambda_1(\lambda_1-\lambda_{\alpha})} + \frac{\tilde{g}^{p\ov{p}} \tilde{g}^{q\ov{q}} | V_1(\tilde{g}_{p\ov{q}})|^2}{\lambda_1} -(1+\ve)\sum_{i\not\in I}\frac{\tilde{g}^{i\ov{i}} | e_i (\varphi_{V_1 V_1})|^2}{\lambda_1^2}\\
&\geq
- 6\ve A^2e^{-2A\vp}  \tilde{g}^{i\ov{i}}|e_i(\vp)|^2  - 6\ve (h')^2 \sum_{i\not\in I}\tilde{g}^{i\ov{i}} |e_i | \partial \varphi|^2_g|^2 - \frac{C}{\ve}\sum_{i}\tilde{g}^{i\ov{i}},
\end{split}
\]
for a uniform constant $C>0$, assuming without loss of generality that at $x_0$ we have $\lambda_1\geq \frac{C}{\ve^3}$.
\end{lemma}
\begin{proof}  It suffices to show that, at $x_0$,
\begin{equation} \label{suffices}
\begin{split}
(2-\ve) \sum_{\alpha >1} &\tilde{g}^{i\ov{i}} \frac{ |e_i (\varphi_{V_{\alpha} V_1})|^2}{\lambda_1(\lambda_1-\lambda_{\alpha})} + \frac{\tilde{g}^{p\ov{p}} \tilde{g}^{q\ov{q}} | V_1(\tilde{g}_{p\ov{q}})|^2}{\lambda_1} -(1+\ve)\sum_{i\not\in I}\frac{\tilde{g}^{i\ov{i}} | e_i (\varphi_{V_1 V_1})|^2}{\lambda_1^2} \\
\geq {} & -3\ve \sum_{i\not\in I}\frac{\tilde{g}^{i\ov{i}} | e_i (\varphi_{V_1 V_1})|^2}{\lambda_1^2}-\frac{C}{\ve}\sum_{i}\tilde{g}^{i\ov{i}},
\end{split}
\end{equation}
since, using $d \hat{Q}|_{x_0}=0$,
\[ \begin{split}
-3\ve\sum_{i\not\in I}\frac{\tilde{g}^{i\ov{i}} | e_i (\varphi_{V_1 V_1})|^2}{\lambda_1^2}&=- 3\ve\sum_{i\not\in I}\tilde{g}^{i\ov{i}} | Ae^{-A\vp} e_i(\varphi) - h' e_i (| \partial \varphi|^2_g)|^2 \\
 &\ge - 6\ve A^2e^{-2A\vp}  \tilde{g}^{i\ov{i}}|e_i(\vp)|^2  - 6\ve (h')^2 \sum_{i\not\in I}\tilde{g}^{i\ov{i}} |e_i | \partial \varphi|^2_g|^2.
\end{split}
\]

We now prove (\ref{suffices}).  At  $x_0$ we have
\begin{equation}\label{tild}
\begin{split}
0<\ti{g}_{\ti{1}\ov{\ti{1}}}= {} &\ti{g}(\ti{e}_1,\ov{\ti{e}}_1)=g(\ti{e}_1,\ov{\ti{e}}_1)+\ti{e}_1\ov{\ti{e}}_1(\vp)-[\ti{e}_1,\ov{\ti{e}}_1]^{(0,1)}(\vp)\\
= {} & 1+\frac{1}{2}(V_1V_1(\vp)+(JV_1) (JV_1)(\vp)+\mn [V_1,JV_1](\vp))-[\ti{e}_1,\ov{\ti{e}}_1]^{(0,1)}(\vp)\\
={} & 1+\frac{1}{2}(\vp_{V_1V_1}+\vp_{JV_1JV_1}+(\nabla_{V_1}V_1)(\vp)+(\nabla_{JV_1}JV_1)(\vp)\\
&\ \ \ \ +\mn [V_1,JV_1](\vp))-[\ti{e}_1,\ov{\ti{e}}_1]^{(0,1)}(\vp)\\
 \le {} & \frac{1}{2}\left(\lambda_1+\sum_{\alpha>1}\lambda_\alpha\mu_\alpha^2\right)+C,
\end{split}
\end{equation}
using that at $x_0$ we have, recalling (\ref{definePhi}) and (\ref{defnmu}),
$$\vp_{V_1 V_1}=g(\Phi(V_1),V_1)=\lambda_1,\quad \vp_{JV_1 JV_1}=g(\Phi(JV_1),JV_1)+ B(JV_1,JV_1)=\sum_{\alpha>1}\lambda_\alpha\mu_\alpha^2+1.$$
The proof splits into two cases.\\

\noindent
{\bf Case (i). }Assume that at $x_0$ we have
\begin{equation}\label{crux}
\frac{1}{2}\left(\lambda_1+\sum_{\alpha>1}\lambda_\alpha\mu_\alpha^2\right)\geq (1-\ve)\ti{g}_{\tilde{1}\ov{\tilde{1}}}>0.
\end{equation}
Then Lemma \ref{lemmatres} gives
\[\begin{split}
(2-\ve) \sum_{\alpha >1} &\tilde{g}^{i\ov{i}} \frac{ |e_i (\varphi_{V_{\alpha} V_1})|^2}{\lambda_1(\lambda_1-\lambda_{\alpha})} + \frac{\tilde{g}^{p\ov{p}} \tilde{g}^{q\ov{q}} | V_1(\tilde{g}_{p\ov{q}})|^2}{\lambda_1} -(1+\ve)\sum_{i\not\in I}\frac{\tilde{g}^{i\ov{i}} | e_i (\varphi_{V_1 V_1})|^2}{\lambda_1^2}\\
&\geq
\sum_{i\not\in I}\sum_{\alpha >1} \frac{\tilde{g}^{i\ov{i}}}{\lambda_1^2}\left(\frac{(2-\ve)\lambda_1}{\lambda_1-\lambda_{\alpha}} |e_i (\varphi_{V_{\alpha} V_1})|^2\right)
+2\sum_{q\in I}\sum_{i\not\in I}\tilde{g}^{i\ov{i}}\tilde{g}^{q\ov{q}}\frac{| V_1(\tilde{g}_{i\ov{q}})|^2}{\lambda_1}\\
&-3\ve \sum_{i\not\in I}\frac{\tilde{g}^{i\ov{i}} | e_i (\varphi_{V_1 V_1})|^2}{\lambda_1^2}
-(1+\gamma)\left(\lambda_1+\sum_{\alpha>1}\lambda_\alpha\mu_\alpha^2\right)\sum_{q\in I}\sum_{i\not\in I}\frac{\tilde{g}^{i\ov{i}}\ti{g}^{q\ov{q}}}{\lambda_1^2} |V_1(\ti{g}_{i\ov{q}})|^2\\
&-\frac{C}{\ve}\sum_{i}\tilde{g}^{i\ov{i}}-(1-\ve)\left(1+\frac{1}{\gamma}\right)\left(\lambda_1-\sum_{\alpha>1}\lambda_\alpha\mu_\alpha^2\right)
\sum_{i\not\in I}\sum_{\alpha>1}\frac{\tilde{g}^{i\ov{i}}}{\lambda_1^2}\frac{|e_i(\vp_{V_\alpha V_1})|^2}{\lambda_1-\lambda_\alpha},
\end{split}\]
and we can now choose
$$\gamma=\frac{\lambda_1-\sum_{\alpha>1}\lambda_\alpha\mu_\alpha^2}{\lambda_1+\sum_{\alpha>1}\lambda_\alpha\mu_\alpha^2}>0,$$
so that the second and fourth term on the right hand side cancel each other, while the first term  dominates the last one, and this establishes \eqref{suffices}.\\

\noindent
{\bf Case (ii). }Assume on the other hand that at $x_0$ we have
\begin{equation}\label{crux2}
\frac{1}{2}\left(\lambda_1+\sum_{\alpha>1}\lambda_\alpha\mu_\alpha^2\right)< (1-\ve)\ti{g}_{\tilde{1}\ov{\tilde{1}}}.
\end{equation}
Then \eqref{tild} gives
$$\ti{g}_{\tilde{1}\ov{\tilde{1}}}\leq \frac{1}{2}\left(\lambda_1+\sum_{\alpha>1}\lambda_\alpha\mu_\alpha^2\right)+C\leq (1-\ve)\ti{g}_{\tilde{1}\ov{\tilde{1}}}+C,$$
and so
\begin{equation}\label{st1}
\ti{g}_{\tilde{1}\ov{\tilde{1}}}\leq \frac{C}{\ve}.
\end{equation}
In general, \eqref{tild} implies that
$$\lambda_1+\sum_{\alpha>1}\lambda_\alpha\mu_\alpha^2\geq -C,$$
and so
$$0<\lambda_1-\sum_{\alpha>1}\lambda_\alpha\mu_\alpha^2\leq 2\lambda_1+C\leq (2+2\varepsilon^2)\lambda_1,$$
as long as $2\lambda_1 \ge C/\ve^2$.

We now choose
$$\gamma=\frac{1}{\ve^2},$$
where $\ve$ still has to be chosen.
 This gives
\[\begin{split}
(1-\ve)\left(1+\frac{1}{\gamma}\right)\left(\lambda_1-\sum_{\alpha>1}\lambda_\alpha\mu_\alpha^2\right)&\leq
2(1-\ve)\left(1+\frac{1}{\gamma}\right)(1+\ve^2)\lambda_1\\
&=2(1-\ve)\left(1+\ve^2\right)^2\lambda_1\leq (2-\ve)\lambda_1,
\end{split}\]
if $\ve \le \frac{1}{6}$,
and from Lemma \ref{lemmatres} we get
\[\begin{split}
(2-\ve) \sum_{\alpha >1} &\tilde{g}^{i\ov{i}} \frac{ |e_i (\varphi_{V_{\alpha} V_1})|^2}{\lambda_1(\lambda_1-\lambda_{\alpha})} + \frac{\tilde{g}^{p\ov{p}} \tilde{g}^{q\ov{q}} | V_1(\tilde{g}_{p\ov{q}})|^2}{\lambda_1} -(1+\ve)\sum_{i\not\in I}\frac{\tilde{g}^{i\ov{i}} | e_i (\varphi_{V_1 V_1})|^2}{\lambda_1^2}\\
&\geq
2\sum_{q\in I}\sum_{i\not\in I}\tilde{g}^{i\ov{i}}\tilde{g}^{q\ov{q}}\frac{| V_1(\tilde{g}_{i\ov{q}})|^2}{\lambda_1}-3\ve \sum_{i\not\in I}\frac{\tilde{g}^{i\ov{i}} | e_i (\varphi_{V_1 V_1})|^2}{\lambda_1^2}\\
&-(1-\ve)\left(1+\frac{1}{\ve^2}\right)\ti{g}_{\ti{1}\ov{\ti{1}}}\sum_{q\in I}\sum_{i\not\in I}\frac{2\tilde{g}^{i\ov{i}}\ti{g}^{q\ov{q}}}{\lambda_1^2} |V_1(\ti{g}_{i\ov{q}})|^2-\frac{C}{\ve}\sum_{i}\tilde{g}^{i\ov{i}}.
\end{split}\]
But thanks to \eqref{st1} we have
$$(1-\ve)\left(1+\frac{1}{\ve^2}\right)\ti{g}_{\ti{1}\ov{\ti{1}}}\leq (1-\ve)\left(1+\frac{1}{\ve^2}\right)\frac{C}{\ve}\leq \frac{C}{\ve^3},$$
and so
\[\begin{split}
(2-\ve) \sum_{\alpha >1} &\tilde{g}^{i\ov{i}} \frac{ |e_i (\varphi_{V_{\alpha} V_1})|^2}{\lambda_1(\lambda_1-\lambda_{\alpha})} + \frac{\tilde{g}^{p\ov{p}} \tilde{g}^{q\ov{q}} | V_1(\tilde{g}_{p\ov{q}})|^2}{\lambda_1} -(1+\ve)\sum_{i\not\in I}\frac{\tilde{g}^{i\ov{i}} | e_i (\varphi_{V_1 V_1})|^2}{\lambda_1^2}\\
\geq {} &
\left(\lambda_1-\frac{C}{\ve^3}\right)2\sum_{q\in I}\sum_{i\not\in I}\tilde{g}^{i\ov{i}}\tilde{g}^{q\ov{q}}\frac{| V_1(\tilde{g}_{i\ov{q}})|^2}{\lambda_1^2}-3\ve \sum_{i\not\in I}\frac{\tilde{g}^{i\ov{i}} | e_i (\varphi_{V_1 V_1})|^2}{\lambda_1^2}-\frac{C}{\ve}\sum_{i}\tilde{g}^{i\ov{i}} \\
\geq {} & -3\ve \sum_{i\not\in I}\frac{\tilde{g}^{i\ov{i}} | e_i (\varphi_{V_1 V_1})|^2}{\lambda_1^2}-\frac{C}{\ve}\sum_{i}\tilde{g}^{i\ov{i}},
\end{split}\]
provided $\lambda_1\geq \frac{C}{\ve^3}$ at $x_0$.   \end{proof}

We now complete the proof of Proposition \ref{secondord}.  Combining Lemmas \ref{lemmae}, \ref{lemmauno} and \ref{lemmacuatro}, and recalling that $h''=2(h')^2$, we have
\[
\begin{split}
0 \ge &  - 6\ve A^2e^{-2A\vp}  \tilde{g}^{i\ov{i}}|e_i(\vp)|^2  - 6\ve (h')^2 \sum_{i\not\in I}\tilde{g}^{i\ov{i}} |e_i | \partial \varphi|^2_g|^2 \\ {} &- 2 (h')^2 \sum_{i\in I}\tilde{g}^{i\ov{i}} |e_i | \partial \varphi|^2_g|^2
 + \frac{h'}{2} \sum_k \tilde{g}^{i\ov{i}} (| e_i e_k (\varphi)|^2 + |e_i \ov{e}_k(\varphi)|^2)  \\ {} & +h'' \tilde{g}^{i\ov{i}} |e_i | \partial \varphi|^2_g|^2 + (Ae^{-A\vp}-\frac{C_1}{\ve}) \sum_i \tilde{g}^{i\ov{i}}+A^2e^{-A\vp}\tilde{g}^{i\ov{i}} |e_i (\varphi)|^2 - Ane^{-A\vp}\\
&\geq \left(Ae^{-A\vp}-\frac{C_1}{\ve}\right)\sum_{i}\tilde{g}^{i\ov{i}}+\frac{h'}{2} \sum_k \tilde{g}^{i\ov{i}} (| e_i e_k (\varphi)|^2 + |e_i \ov{e}_k(\varphi)|^2)\\
&+(A^2e^{-A\vp}-6\ve A^2e^{-2A\vp})\tilde{g}^{i\ov{i}} |e_i (\varphi)|^2 - Ane^{-A\vp},
\end{split}
\]
as long as $\ve\leq \frac{1}{6},$ where $C_1$ is uniform constant.
We now make our choices of $A$ and $\ve$.  First choose
$$A=6C_1+1,$$
and, at the expense of increasing $C_1$, we may assume that $$A \ge \max (C_0 +1, 6n \sup_M |\partial \varphi|^2_g)$$ which we needed earlier in (\ref{A1}) and Lemma \ref{lemmauno}.  Next pick
$$\ve=\frac{e^{A\vp(x_0)}}{6} \le \frac{1}{6},$$
so that
$$Ae^{-A\vp(x_0)}-\frac{C_1}{\ve}= e^{-A\vp(x_0)}\geq 1,\quad A^2e^{-A\vp(x_0)}-6\ve A^2e^{-2A\vp(x_0)}=0.$$
Note that now that the values of $\ve$ and $A$ have been fixed, we may indeed assume without loss of generality that at $x_0$ we have
$\lambda_1\geq\frac{C}{\ve^3}$ and $\lambda_1\geq \frac{C_A}{\ve}$.
We conclude that at $x_0$ we have
\begin{equation}\label{dumb2}
\sum_{i}\tilde{g}^{i\ov{i}}+\frac{h'}{2} \sum_k \tilde{g}^{i\ov{i}} (| e_i e_k (\varphi)|^2 + |e_i \ov{e}_k(\varphi)|^2)\leq C.
\end{equation}
From this it follows that at $x_0$ we have $\sum_{i}\tilde{g}^{i\ov{i}}\leq C,$ and so using again \eqref{dumb} we see that $\ti{g}$ is uniformly equivalent to $g$, and from \eqref{dumb2} we have
$\sum_{i,k} (| e_i e_k (\varphi)|^2 + |e_i \ov{e}_k(\varphi)|^2)\leq C$, which (as in \eqref{nabb}, \eqref{nabb3}, \eqref{nabb2}) implies that
$\sum_{\alpha,\beta}|\nabla^2_{\alpha\beta}\vp|\leq C,$ and so $\lambda_1(x_0)\leq C$, and hence $Q(x_0)\leq C$, as desired.
\end{proof}

\section{Completion of the proofs of the Main Theorems} \label{sectioncomplete}

In this section, we complete the proofs of the main results stated in the introduction.

\begin{proof}[Proof of Theorem \ref{main2}]
Thanks to Propositions \ref{uniform estimate} and \ref{secondord}, we can obtain {\em a priori} $C^{2,\alpha}$ estimates from the main result of \cite{TWWY}, for some $0< \alpha <1$ (and in fact for all $0<\alpha<1$ by \cite{Chu2}). Higher order estimates follow after differentiating the equation and applying the usual
bootstrapping method.
\end{proof}

Our main result now follows easily:

\begin{proof}[Proof of Theorem \ref{main}]  Now we have Theorem \ref{main2}, we can follow closely the arguments of \cite{TW0}.   Indeed, we will see that the $(1,1)$ form $\mn\de\db\vp=\frac{1}{2}(d(Jd\vp))^{(1,1)}$ has the right properties for this part of the  proof to go through essentially unchanged from the integrable case.  We include the brief arguments  for the sake of completeness.
  Consider the family of equations
\begin{equation} \label{family}
(\omega + \ddbar \varphi_t)^n = e^{tF+b_t} \omega^n, \quad \omega + \ddbar \varphi_t >0,
\end{equation}
for $t \in [0,1]$, where $b_t \in \mathbb{R}$.
For a fixed $\alpha \in (0,1)$, we define $\mathcal{T}$ to be the set of $t \in [0,1]$ such that (\ref{family}) admits a solution $(\varphi_t, b_t) \in C^{3,\alpha}(M) \times \mathbb{R}.$  Clearly $0 \in \mathcal{T}$.  We will show that $\mathcal{T}$ is open and closed, since this will show that $1\in \mathcal{T}$, and then adding a constant to $\vp_1$ we will obtain a solution of \eqref{ma}.

 To prove that $\mathcal{T}$ is open, fix $\hat{t} \in \mathcal{T}$, and write $\hat{\omega} := \omega + \ddbar \varphi_{\hat{t}}$.
From Theorem \ref{gaud}, there exists a smooth $\hat{u}$ such that $e^{\hat{u}} \hat{\omega}$ is Gauduchon, and we may assume that  $\int_M e^{(n-1)\hat{u}} \hat{\omega}^n =1$.  Define a map $\Psi$ by
$$\Psi(\psi) = \log \frac{(\hat{\omega}+\ddbar \psi)^n}{\hat{\omega}^n} - \log \left( \int_M e^{(n-1)\hat{u}} (\hat{\omega} + \ddbar \psi)^n \right),$$
which takes $\psi \in C^{3,\alpha}(M)$ with $\int_M \psi e^{(n-1)\hat{u}} \hat{\omega}^n=0$ and $\hat{\omega} + \ddbar \psi>0$ to a $C^{1,\alpha}(M)$ function $h$ satisfying $\int_M e^h e^{(n-1) \hat{u}} \hat{\omega}^n =1$.  For $t$ close to $\hat{t}$ we wish to solve
\begin{equation} \label{Psit}
\Psi(\psi_t) = (t-\hat{t} \,) F - \log \left( \int_M e^{(t-\hat{t}\, )F} e^{(n-1) \hat{u}} \hat{\omega}^n \right).
\end{equation}
Observe that $\Psi(0)=0$, and that the linearization of $\Psi$ at $\psi=0$ is given by the canonical Laplacian (see (\ref{canonical})) of $\hat{\omega}$,
\begin{equation} \label{eta}
\eta \mapsto \Delta_{\hat{\omega}}^C \eta,
\end{equation}
since $\int_M \ddbar \eta \wedge e^{(n-1)\hat{u}} \hat{\omega}^{n-1} =0$ (which follows from integration by parts, as in (\ref{ibp}), and the fact that $e^{\hat{u}} \hat{\omega}$ is Gauduchon).   By Theorem \ref{solvepoi},  the operator (\ref{eta}) is an invertible map of the tangent  spaces, and so by the Inverse Function Theorem we obtain $\psi_t$ solving (\ref{Psit}) for $t$ close to $\hat{t}$.  Then $\varphi_t = \varphi_{\hat{t}} + \psi_t$  solves (\ref{family}) for $t$ close to $\hat{t}$, for  $b_t \in \mathbb{R}$, showing that $\mathcal{T}$ is open.

To prove that $\mathcal{T}$ is closed we will apply Theorem \ref{main2}.  Suppose $\varphi_t \in C^{3,\alpha}(M)$ solves (\ref{family}).  Then,  as above, we can differentiate the equation and bootstrap to obtain that $\varphi_t$ is smooth.
Applying the maximum principle (\ref{localmax}) to $\varphi_t$ in (\ref{family}) we obtain the bound $|b_t| \le \sup_M |F|$.  We may add a $t$-dependent constant to $\varphi_t$ solving (\ref{family})  so that $\sup_M \varphi_t=0$.
Now we  apply Theorem \ref{main2} to $\varphi_t$  to obtain $C^{3, \alpha}$ (in fact, $C^{\infty}$) estimates  which are independent of $t$.   This shows that $\mathcal{T}$ is closed.  Hence we have proved the existence of a $C^{3, \alpha}$, and hence smooth, solution of (\ref{ma}).

It remains to prove uniqueness.  Assume that we have two solutions $(\varphi, b)$ and $(\varphi', b')$ of (\ref{ma}).  Writing $\theta = \varphi - \varphi'$, we have
$$\frac{(\omega + \ddbar \varphi' + \ddbar \theta)^n}{(\omega + \ddbar \varphi')^n} = e^{b-b'}.$$
Applying (\ref{localmax}) at the extrema  of $\theta$ we obtain $b=b'$.  Then
$$(\omega + \ddbar \varphi)^n = (\omega+ \ddbar \varphi')^n$$
and so
$$\ddbar \theta \wedge \sum_{i=0}^{n-1} (\omega+ \ddbar \varphi)^i \wedge (\omega + \ddbar \varphi')^{n-1-i} =0.$$
The strong maximum principle (together with the fact that $\sup_M\vp=\sup_M\vp'=0$) implies that $\theta =0$, namely $\varphi=\varphi'$.  This completes the proof.
\end{proof}

\end{document}